\documentclass[11pt]{amsart} 
\usepackage{amsmath, amssymb, verbatim, amscd, amsthm, bm}
\usepackage[normalem]{ulem}
\usepackage[driverfallback=dvipdfm]{hyperref}
\usepackage[dvipdfm]{graphicx}
\usepackage{color,graphicx,shortvrb}
\usepackage{enumerate}

\newtheorem{thm}{Theorem}

\newtheorem{lem}{Lemma}[section] 
\newtheorem{prop}[lem]{Proposition}
\newtheorem{cor}[thm]{Corollary} 
\newtheorem{claim}{Claim}

\theoremstyle{definition}
\newtheorem{defn}{Definition}[section]

\newtheorem*{nota}{Notation}
 
\numberwithin{equation}{section}

 \newcommand{\R}{\mathbb R} 
 \newcommand{\C}{\mathbb C} 
 \newcommand{\T}{\mathbb T}

 \newcommand{\la}{\lambda} 
 \newcommand{\D}{\displaystyle} 
 \newcommand{\A}{\mathcal{A}} 
 \newcommand{\B}{\mathfrak{B}} 

 \newcommand{\q}{\quad} 
 \newcommand{\qq}{\qquad} 
 \newcommand{\ph}[1]{\phantom{#1}} 
 \newcommand{\Ch}[1]{\operatorname{Ch}(#1)}
 \newcommand{\si}{\sigma} 
 \newcommand{\ov}{\overline} 
 \newcommand{\fr}{\frac} 
 
 \newcommand{\set}[1]{\{ #1 \}} 
  
 \newcommand{\al}{\alpha}

 \renewcommand{\Re}{{\rm Re}\,} 
 \renewcommand{\Im}{{\rm Im}\,} 
  
 \renewcommand{\d}{\delta} 

 \newcommand{\dphi}[1]{\d_{\Phi(#1)}}

 \newcommand{\unit}{\mathbf{1}}
 \newcommand{\tio}{\widetilde{\mathbf{1}}}
 \newcommand{\muo}{\mu_1}

 \newcommand{\h}{\eta}

 \DeclareSymbolFont{largesymbol}{OMX}{yhex}{m}{n}
 \newcommand{\wh}{\widehat}
 \newcommand{\wt}{\widetilde}

 \newcommand{\pio}{\pi_1}
 \newcommand{\pit}{\pi_2}
 \newcommand{\varphio}{\varphi_1}
 \newcommand{\Di}{\mathbb D}
 \newcommand{\Db}{\overline{\mathbb D}}
 \newcommand{\M}{\mathcal M}
 
 \newcommand{\clk}[2]{\operatorname{cl}_{#1}(#2)}

\newcommand{\Aos}{\TA_1^*}
\newcommand{\po}{\p_1}
\newcommand{\pt}{\p_2}

\newcommand{\xo}{\x_1}
\newcommand{\xt}{\x_2}
\newcommand{\yo}{\y_1}
\newcommand{\yt}{\y_2}
\newcommand{\xio}{\xi_1}
\newcommand{\xit}{\xi_2}
\newcommand{\del}{\Delta}
\newcommand{\ext}{\operatorname{ext}(\Aos)}
\newcommand{\laz}{\la_0}
\newcommand{\lao}{\la_1}

\newcommand{\alo}{\al_1}

\renewcommand{\B}{B}
\renewcommand{\A}{A}
\newcommand{\X}{X}
\newcommand{\Y}{Y}
 \newcommand{\TF}{\widetilde{f}}

\newcommand{\XYT}{\X\times\Y\times\T}

\newcommand{\TA}{\widetilde{\A}}
 \newcommand{\TG}{\widetilde{g}}
\newcommand{\chta}{\Ch{\TA}}
 \newcommand{\cha}{\Ch{\A}}
 \newcommand{\chb}{\Ch{\B}} 
 \renewcommand{\TH}{\wt{h}}
 \newcommand{\phio}{\phi}
 
 \newcommand{\psio}{\psi}
 \newcommand{\nuz}{\nu_0}
 \newcommand{\nuo}{\nu_1}
 \newcommand{\nut}{\nu_2}
 
 \newcommand{\MD}{\mathcal{D}}
 \newcommand{\Z}{\mathcal{Z}}
 \newcommand{\bta}{\MD\times\T}
 \newcommand{\e}{\epsilon} 
 \newcommand{\ez}{\epsilon_0} 
 \newcommand{\eo}{\epsilon_1} 
 \newcommand{\et}{\epsilon_2} 
 \newcommand{\ve}{\varepsilon}
 \newcommand{\PAp}{\PA_+}
 \newcommand{\PAm}{\PA_-}
 \newcommand{\K}{K}
 \newcommand{\PA}{\partial\A}
 \newcommand{\PB}{\partial\B}
 \newcommand{\PBp}{\PB_+}
 \newcommand{\PBm}{\PB_-} 
 \newcommand{\V}[1]{\left\Vert#1\right\Vert}
 \newcommand{\VL}[1]{\Vert#1\Vert}
 \newcommand{\VK}[1]{\left\Vert#1\right\Vert_{\K}}
 \newcommand{\VA}[1]{\V{#1}}
 \newcommand{\VLA}[1]{\VL{#1}}
\newcommand{\Vinf}[1]{\left\Vert#1\right\Vert_\infty}
 \newcommand{\VLinf}[1]{\Vert#1\Vert_\infty}
 \newcommand{\VX}[1]{\left\Vert#1\right\Vert_{\X}}

 \renewcommand{\VX}[1]{\left\Vert#1\right\Vert_{\X}}
 \newcommand{\VY}[1]{\left\Vert#1\right\Vert_{\Y}}
 \newcommand{\VS}[1]{\left\Vert #1 \right\Vert_{(\Sigma)}}
 \newcommand{\Vs}[1]{\left\Vert #1 \right\Vert_{(\sigma)} }
 \newcommand{\x}{x}
 \newcommand{\y}{y}
 \newcommand{\p}{p}
 \newcommand{\pz}{\p_0}
 \newcommand{\xz}{\x_0}
 \newcommand{\yz}{\y_0}

\usepackage{ulem}
\usepackage{marginnote}
\usepackage{geometry}

\def\rsout#1{\color{red} \xout{#1}\color{black} \marginnote{\color{red}
$\leftrightarrow$ \color{black}}}

\begin{document}


\title[Isometries on function spaces]{Surjective isometries on function spaces with derivatives}

\author{M.G. Cabrera-Padilla}
\address{Departamento de Matem\'{a}ticas,
Universidad de Almer\'{i}a,
04120,
Almer\'{i}a, Spain}
\email{m\_gador@hotmail.com}

\author[A.~Jim\'{e}nez-Vargas]{A. Jim\'{e}nez-Vargas}
\address{Departamento de Matem\'{a}ticas,
Universidad de Almer\'{i}a,
04120,
Almer\'{i}a, Spain}
\email{ajimenez@ual.es}

\author[T.~Miura]{Takeshi Miura}
\address{Department of Mathematics,
Faculty of Science, Niigata University,
Niigata 950-2181, Japan}
\email{miura@math.sc.niigata-u.ac.jp}

\author[M.~Villegas-Vallecillos]{Mois\'{e}s Villegas-Vallecillos}
\address{Departamento de Matem\'{a}ticas,
Universidad de C\'{a}diz,
11510 Puerto Real, Spain}
\email{moises.villegas@uca.es}

\begin{abstract}
Let $\A$ be a complex Banach space with a norm $\VA{f}=\VX{f}+\VY{d(f)}$ for $f\in\A$, where $d$ is a complex linear map from $\A$ onto a Banach space $\B$, and $\VK{\cdot}$ represents the supremum norm on a compact Hausdorff space $\K$. In this paper, we characterize surjective isometries on $(\A,\VA{\cdot})$, which may be nonlinear. This unifies former results on surjective isometries between specific function spaces.
\end{abstract}

\maketitle

\section{Introduction and the main results}

Given a real or complex Banach space $(E,\V{\cdot})$, a mapping $T\colon E\to E$ is referred to as an \textit{isometry} if $\V{T(a)-T(b)}=\V{a-b}$ for all $a,b\in E$. The earliest known research on isometries can be traced back to the 1930s. Let $C_\R(\X)$ denote the real Banach space of all real-valued continuous functions on a compact Hausdorff space $\X$, equipped with the supremum norm $\VX{\cdot}$. In a seminal work, Banach \cite{ban} characterized surjective, possibly nonlinear, isometries from $C_\R(\X)$ onto $C_\R(\Y)$. Since then, a substantial body of research on isometries has been developed. For example, Rao and Roy \cite{rao} described surjective \textit{complex linear} isometries on the Banach space $C^1([0,1])$ of all continuously differentiable functions on the closed unit interval $[0,1]$ with respect to the norm $\VS{f}=\Vinf{f}+\Vinf{f'}$ for $f\in C^1([0,1])$, where $\Vinf{\cdot}$ denotes the supremum norm on $[0,1]$. Novinger and Oberlin \cite{nov} conducted research on \textit{complex linear} isometries on Banach spaces of analytic functions on the open unit disc $\Di$. To be more precise, they clarified the structure of isometries on Banach spaces of all analytic functions on $\Di$ whose derivatives belong to the Hardy space $H^p(\Di)$ for $1\leq p<\infty$. The case of $p=\infty$ was resolved by the third author of this paper and Niwa \cite{miu4} for surjective isometries. The objective of this paper is to examine the properties of surjective, which may or may not be linear, isometries on function spaces that involve derivatives valued in Banach spaces.

Let $C(\K)$ be the complex Banach space of all continuous complex-valued functions defined on a compact Hausdorff space $\K$ with the supremum norm $\VK{\cdot}$. Let $\A$ be a linear subspace of $C(\K)$ and $\B$ a closed linear subspace of $(C(\Y),\V{\cdot}_Y)$ for some compact Hausdorff space $\Y$. For a non-empty compact subset $\X$ of $\K$, we denote by $\cha$ the Choquet boundary for $\A$ as a normed linear subspace of $(C(\X),\VX{\cdot})$. The Choquet boundary for $\A$ is defined as the set of all $\x\in\X$ for which the point evaluation functional $\d_{\x}\colon\A\to\C$, defined by $\d_{\x}(f)=f(\x)$ for all $f\in\A$, is an extreme point of the closed unit ball of the dual space of $\A$. The Choquet boundary for $\B$ is denoted by $\chb$. The closures of $\cha$ and $\chb$ in $\X$ and $\Y$, respectively, are denoted by $\PA$ and $\PB$. These are called the Shilov boundaries for $\A$ and $\B$, respectively. Furthermore, we assume that $\A$ and $\B$ satisfy the following conditions:
\begin{enumerate}
\item\label{(1)} There exists a surjective, complex linear map $d\colon\A\to\B$.
\item\label{(2)} The linear space $\A$ is complete with respect to the norm $\VA{\cdot}$, defined as $\VA{f}=\VX{f}+\VY{d(f)}$ for $f\in\A$.
\item\label{(3)} $\A$ contains the constant function $\unit\in C(\X)$, which maps $\x$ to $1$ for all $\x\in\X$. 
\item\label{(4)} The kernel, $\ker(d)$, of the operator $d$ satisfies $\ker(d)=\set{z\unit\in\A:z\in\C}$.
\item There exists a dense subset $\MD$ of $\PA\times\PB$ with the following properties:
\begin{enumerate}
\item\label{(5)} For each $\ve>0$, $\yz\in\PB$ with $(\xz,\yz)\in\MD$ for some $\xz\in\PA$ and an open set $N_0\subset\Y$ with $\yz\in N_0$, there exists $f_0\in\A$ such that $\VX{f_0}<\ve$, $d(f_0)(\yz)=1=\VY{d(f_0)}$ and $|d(f_0)|<\ve$ on $\PB\setminus N_0$.
\item\label{(6)} 
For each $\ve>0$, $(\xz,\yz)\in\MD$ and an open set $W_0\subset\X$ with $\xz\in W_0$, there exists $g_0\in\A$ such that $g_0(\xz)=1=\VX{g_0}$, $|g_0|<\ve$ on $\PA\setminus W_0$ and $|d(g_0)(\yz)|<\ve$.
\end{enumerate}
\item\label{(7)} Let $(\x_k,\y_k)\in\PA\times\PB$ for $k=0,1,2$.
\begin{enumerate}
\item\label{(7a)} If $\xz\not\in\set{\xo,\xt}$, there exists $h_0\in\A$ such that $h_0(\xz)=1$, $h_0(\x_j)=0$ for $j=1,2$, and $d(h_0)(\y_k)=0$ for $k=0,1,2$.
\item\label{(7b)} If $\yz\not\in\set{\yo,\yt}$, there exists $h_1\in\A$ such that $h_1(\x_k)=0$ for $k=0,1,2$, $d(h_1)(\yz)=1$ and $d(h_1)(\y_j)=0$ for $j=1,2$.
\item\label{(7c)} There exist $h_2,h_3\in\A$ such that $h_2(\xz)=1$, $d(h_2)(\yz)=0$, $h_3(\xz)=0$ and $d(h_3)(\yz)=1$.
\end{enumerate}
\end{enumerate}

In this context, it is important to highlight that the linear space $\A$ has two norms: $\VX{\cdot}$ and $\VA{\cdot}$. It is not necessarily the case that $\A$ is closed with respect to the topology induced by $\VX{\cdot}$. In contrast, $\A$ is a complex Banach space equipped with $\VA{\cdot}$ by definition. 

The following examples satisfy the conditions in \eqref{(1)} through \eqref{(7)}: the set $C^1([0,1])$ of all continuously differentiable functions on $[0,1]$; the set of all continuous extensions to $\Db$ of all analytic functions on $\Di$, which can be extended to continuous functions on $\Db$; and the set of all Gelfand transforms of analytic functions on $\Di$ whose derivatives are bounded on $\Di$. In section \ref{sect5}, we shall discuss these examples in detail.

Let $\T$ denote the unit circle of the complex number field $\C$. Given $z\in\C$, $\overline{z}$, $\Re(z)$ and $\Im(z)$ stand for the complex conjugate, the real part and the imaginary part of $z$, respectively.

The main results of this paper are as follows:

\begin{thm}\label{thm1}
Let $\A$ and $\B$ be Banach spaces with the properties from \eqref{(1)} through \eqref{(7)}. If $T\colon\A\to\A$ is a surjective isometry, then there exist a constant $c\in\T$, homeomorphisms
$\phi\colon\PA\to\PA$
and $\psi\colon\PB\to\PB$,
a continuous function $u\colon\PB\to\T$
and closed and open, possibly empty, subsets
$\PAp$ of $\PA$
and $\PBp$ of $\PB$
such that
\begin{align*}
&T(f)(\x)-T(0)(\x)
=
\begin{cases}
cf(\phi(\x))&\x\in\PAp\\[1mm]
c\ov{f(\phi(\x))}&\x\in\PA\setminus\PAp
\end{cases},
\\[1mm]
&d(T(f)-T(0))(\y)
=
\begin{cases}
u(\y)d(f)(\psi(\y))&\y\in\PBp\\[1mm]
u(\y)\ov{d(f)(\psi(\y))}&\y\in\PB\setminus\PBp
\end{cases}
\end{align*}
for all $f\in\A$, $\x\in\PA$
and $y\in\PB$.

Conversely, in the event that the operators
$T$ and $d\circ T$ are of the aforementioned forms,
then $T$ is an isometry.
\end{thm}

\begin{cor}\label{cor2}
Let $\A$, $\B$ and $T$ be as in Theorem~\ref{thm1}. If $X$ is a one point set $\set{\xz}$, then there exist a constant $c\in\T$, a
continuous function
$u\colon\PB\to\T$,
a homeomorphism $\psi\colon\PB\to\PB$ and a closed and open, possibly empty, subset $\PBp$ of $\PB$ such that either 
$$
T(f)(\xz)-T(0)(\xz)=cf(\xz)
$$ 
for all $f\in\A$ or 
$$
T(f)(\xz)-T(0)(\xz)=c\ov{f(\xz)}
$$
for all $f\in\A$, and 
$$
d(T(f)-T(0))(\y)=
\begin{cases}
u(\y)d(f)(\psi(\y))&\y\in\PBp\\[1mm]
u(\y)\ov{d(f)(\psi(\y))}&\y\in\PB\setminus\PBp
\end{cases}
$$
for all $f\in\A$ and $\y\in\PB$.

Conversely, if the operators $T$ and $d\circ T$ are
of the aforementioned forms,
then $T$ is an isometry.
\end{cor}

The following is an outline of the proof. The initial step is to embed the Banach space $(\A,\VA{\cdot})$ isometrically into $(C(\Z),\Vinf{\cdot})$ for some compact Hausdorff space $\Z$,
where $\Vinf{\cdot}$ denotes the supremum norm on the set $\Z$. Let $\TA$ denote the isometric embedding image of $\A$ in $C(\Z)$. The Choquet boundary for $\TA$ is then investigated. When $T$ is a surjective isometry from $\A$ onto itself, it naturally induces a surjective \textit{real linear} isometry $S$ on $\TA$ by utilizing the Mazur--Ulam theorem \cite{maz,nic}. Subsequently, we introduce a kind of adjoint operator, designated as $S_*$, of $S$
on the dual space $\TA^*$ of $\TA$. In the case where $S$ is complex linear, then the adjoint operator of $S$ is well-defined. However, in the context of the real linearity, an adjustment to the definition of the adjoint operator of $S$ is required. We introduce a set $\del$, which is closely related to the set $\ext$ of all extreme points of the closed unit ball $\Aos$ of the dual space $\TA^*$ of $(\TA,\Vinf{\cdot})$. It is shown that the operator $S_*$ preserves the set $\del$. We then demonstrate the form of $S$ by employing the equation $S_*(\del)=\del$. It can be concluded that surjective isometries $T$ on $\A$ are of the form described in Theorem~\ref{thm1} and Corollary~\ref{cor2}.

In the remainder of this paper, we proceed under the assumption that $\A$ and $\B$ satisfy the properties outlined in this section.

\section{An embedding of $\A$ and its Choquet boundary}

We shall introduce a compact Hausdorff space $\Z$ for which the Banach space $\A$ is isometrically embedded into $C(\Z)$. Subsequently, we investigate the Choquet boundary for the isometric image of $\A$. 

For each $f\in\A$, $\x\in\X$, $\y\in\Y$ and $\nu\in\T$, we define
\begin{equation}\label{tilde}
\TF(x,y,\nu)=f(\x)+d(f)(\y)\nu.
\end{equation}
It can be demonstrated that $\TF$ is a continuous complex-valued function defined on a compact Hausdorff space $\XYT$ with respect to the product topology. 
For the sake of simplicity in notation, we set 
$$
\Z=\PA\times\PB\times\T
$$ 
and $\VLinf{\TF}=\sup_{(\x,\y,\nu)\in\Z}|\TF(\x,\y,\nu)|$. We shall prove that
\begin{equation}\label{tildenorm}
\VLinf{\TF}
=\VX{f}+\VY{d(f)}
\end{equation}
for all $f\in\A$. Fix an arbitrary $f\in\A$. For each $(\x,\y,\nu)\in\Z$, we derive from \eqref{tilde} that
\[
|\TF(\x,\y,\nu)|
\leq
|f(\x)|+|d(f)(\y)|
\leq
\VX{f}+\VY{d(f)},
\]
which shows that $\VLinf{\TF}\leq\VX{f}+\VY{d(f)}$. Conversely, there exist $\xo\in\cha$ and $\yo\in\chb$ such that $\VX{f}=|f(\xo)|$ and $\VY{d(f)}=|d(f)(\yo)|$, since $\cha$ and $\chb$ are boundaries
for $\A$ and $\B$, respectively (see \cite[Theorem~2.3.8]{fle1} and \cite[Proposition~2.20]{hat1}). We can choose $\nu_1\in\T$ so that $|f(\xo)|+|d(f)(\yo)|=|f(\xo)+d(f)(\yo)\nu_1|$. In fact, there exist $\lao,\muo\in\T$ such that
$f(\xo)=|f(\xo)|\lao$ and $d(f)(\yo)=|d(f)(\yo)|\muo$.
Let us set $\nu_1=\lao\ov{\muo}$,
which yields $\nu_1\in\T$ and
$d(f)(\yo)\nu_1=|d(f)(\yo)|\lao$.
We conclude that
\[
|f(\xo)+d(f)(\yo)\nu_1|
=\bigl||f(\xo)|\lao+|d(f)(\yo)|\lao\bigr|
=|f(\xo)|+|d(f)(\yo)|,
\]
as claimed. Then we have
\[
\VX{f}+\VY{d(f)}
=|f(\xo)|+|d(f)(\yo)|
=|f(\xo)+d(f)(\yo)\nu_1|=|\TF(\xo,\yo,\nu_1)|
\leq\VLinf{\TF},
\]
which yields the desired result of \eqref{tildenorm}.

For each $f\in\A$, we have that $\VA{f}=\VX{f}+\VY{d(f)}$ by definition. By combining equation \eqref{tildenorm} with the previous equation, we obtain the following result:
\begin{equation}\label{dnorm}
\VA{f}=\VLinf{\TF}
\qq (f\in\A).
\end{equation}
We define a subset $\TA$ of $C(\Z)$ as
\[
\TA=\set{\TF\in C(\Z):f\in\A}.
\]
It can be seen that $\TA$ is a linear subspace of $C(\Z)$ by \eqref{tilde}. Let us define a mapping $U$ from $(\A,\VA{\cdot})$ to $(\TA,\Vinf{\cdot})$ as
\begin{equation}\label{U}
U(f) = \TF
\qq(f\in\A).
\end{equation}
Equations \eqref{tilde} and \eqref{dnorm} show that $U$ is a surjective complex linear isometry. Therefore, we conclude that $\TA$ is a closed linear subspace of $C(\Z)$.

By equation \eqref{tilde} and conditions \eqref{(3)} and \eqref{(4)}, we see that the function $\wt{\unit}\in\TA$ is a constant function with a single value of $1$, defined on $\XYT$.

Let $\xi$ be a bounded linear functional on $\TA$ with the operator norm $\V{\xi}$. By the Hahn--Banach theorem, it can be extended to a bounded linear functional, $\xi_1$, on $C(\Z)$ with $\V{\xi_1}=\V{\xi}$. By applying the Riesz representation theorem, there exists a regular Borel measure, denoted by $\sigma$, on $\Z$ such that $\V{\si}=\V{\xi_1}$ and $\D\xi_1(u)=\int_{\Z}u\,d\si$ for all $u\in C(\Z)$. Here, $\V{\sigma}$ represents the total variation of $\si$. The measure $\si$ is referred to as a representing measure for $\xi$. It should be noted that there is no unique determining factor for this measure, as it is dependent on the Hahn--Banach extension of $\xi$.

By the following three lemmas, we will prove that  any representing measure for a point evaluation, $\d_{\p}$, is the Dirac measure concentrated at the point $\p\in\MD\times\T$. We will begin by demonstrating that each representing measure for $\d_{\p}$ is concentrated on the second coordinate of the point $\p$.

\begin{lem}\label{lem3.1}
Let $\pz=(\xz,\yz,\nuz)\in\bta$ and let $\si$ be a representing measure for $\d_{\pz}$. Then $\sigma(\PA\times\set{\yz}\times\T)=\si(\Z)=1$.
\end{lem}

\begin{proof}
Note that $\si$ is a regular Borel measure such that $\V{\sigma}=\V{\d_{\pz}}$ and $\D\d_{\pz}(\TF)=\int_{\Z} \TF\,d\si$ for every $\TF\in\TA$. Since $\d_{\pz}(\widetilde{\unit})=1=\V{\d_{\pz}}$ by equations \eqref{tilde} and \eqref{tildenorm} with condition \eqref{(4)}, any representing measure for $\d_{\pz}$ is a probability measure.
This is proved in, for example, \cite[p. 81]{bro}.
In particular, we deduce that $\si(\Z)=\V{\si}=\V{\d_{\pz}}=1$.
Let $\ve>0$ and let $N_0$ be an open set in $\Y$ with $\yz\in N_0$. Let us define $N_0^c=\PB\setminus N_0$. We shall prove that $\si(\PA\times N_0^c\times\T)=1$. We need to consider the case when $N_0^c\neq\emptyset$
by the fact that $\si(\PA\times\PB\times\T)=\si(\Z)=1$. Then, there exists $f_0\in\A$ such that \begin{equation}\label{lem3.1.1}
\VX{f_0}<\ve,
\q
d(f_0)(\yz)=1=\VY{d(f_0)}
\qq\mbox{and}\qq
|d(f_0)|<\ve
\q\mbox{on}\q N_0^c
\end{equation}
by condition \eqref{(5)}. We infer from equations \eqref{tilde} and \eqref{lem3.1.1} that 
\begin{equation}\label{lem3.1.2} 
\VLinf{\TF_0}<\ve+1
\qq\mbox{and}\qq
|\TF_0|<2\ve\q\mbox{on}\q\PA\times N_0^c\times\T.
\end{equation}
Having in mind that $\si$ is a representing measure for $\d_{\pz}$, we obtain
\[
\int_{\Z}\TF_0\,d\si
=\d_{\pz}(\TF_0)
=\TF_0(\xz,\yz,\nuz)
=f_0(\xz)+d(f_0)(\yz)\nuz
=f_0(\xz)+\nuz,
\]
which implies that $\D\left|\int_{\Z}\TF_0\,d\si\right|\geq|\nuz|-|f_0(\xz)|>1-\ve$ by \eqref{lem3.1.1}. We derive from \eqref{lem3.1.2} that
\begin{align*}
1-\ve
&<
\left|\int_{\Z}\TF_0\,d\si\right|
\leq\left|\int_{\PA\times N_0\times\T}\TF_0\,d\si\right|
+\left|\int_{\PA\times N_0^c\times\T}\TF_0\,d\si\right|\\
&<(\ve+1)\si(\PA\times N_0\times\T)
+2\ve\si(\PA\times N_0^c\times\T).
\end{align*}
Since $\ve>0$ is chosen arbitrarily, we deduce that $1\leq\si(\PA\times N_0\times\T)$. Since $\si(\PA\times N_0\times\T)\leq\si(\Z)=1$, we obtain that $\si(\PA\times N_0\times\T)=1$. In light of the fact that $\si$ is a regular measure, we conclude that $\si(\PA\times\set{\yz}\times\T)=1$.
\end{proof}

Next, we prove that any representing measure for $\d_{\p}$ is concentrated on the first coordinate of the point $\p\in\MD\times\T$.

\begin{lem}\label{lem3.2}
Let $\pz=(\xz,\yz,\nuz)\in\bta$ and let $\si$ be a representing measure for $\d_{\pz}$. Then $\sigma(\set{\xz}\times\set{\yz}\times\T)=\si(\Z)=1$.
\end{lem}
\begin{proof}
Let $\ve>0$ and let $W_0$ be an arbitrary open set in $\X$ with $\xz\in W_0$. Define $W_0^c=\PA\setminus W_0$. Then, by condition \eqref{(6)}, there exists $g_0\in\A$ that satisfies the following:
\begin{equation}\label{lem3.2.1}
g_0(\xz)=1=\VX{g_0},
\q
|g_0|<\ve
\q\mbox{on}\q W_0^c
\qq\mbox{and}\qq
|d(g_0)(\yz)|<\ve.
\end{equation}
We derive from \eqref{tilde} and \eqref{lem3.2.1} that
\begin{equation}\label{lem3.2.2}
|\TG_0|<1+\ve
\q\mbox{on}\q\PA\times\set{\yz}\times\T
\qq\mbox{and}\qq
|\TG_0|<2\ve
\q\mbox{on}\q W_0^c\times\set{\yz}\times\T.
\end{equation}
The equation
$\si(\PA\times\set{\yz}\times\T)=\si(\Z)=1$ is the conclusion of Lemma~\ref{lem3.1}. Since $\si$ is a representing measure for $\d_{\pz}$, we obtain
\[
\int_{\PA\times\set{\yz}\times\T}\TG_0\,d\si
=\int_{\Z}\TG_0\,d\si
=\d_{\pz}(\TG_0)
=g_0(\xz)+d(g_0)(\yz)\nuz
=1+d(g_0)(\yz)\nuz,
\]
where we have used \eqref{lem3.2.1}. Keeping in mind that $|d(g_0)(\yz)|<\ve$ by \eqref{lem3.2.1}, we infer from \eqref{lem3.2.2} that 
\begin{align*}
1-\ve
&<
\left|\int_{\PA\times\set{\yz}\times\T}\TG_0\,d\si\right|
\leq\left|\int_{W_0\times\set{\yz}\times\T}\TG_0\,d\si\right|
+\left|\int_{W_0^c\times\set{\yz}\times\T}\TG_0\,d\si\right|\\
&<(1+\ve)\si(W_0\times\set{\yz}\times\T)
+2\ve\si(W_0^c\times\set{\yz}\times\T).
\end{align*}
Letting $\ve\to+0$, we obtain the result that $1\leq\si(W_0\times\set{\yz}\times\T)\leq\si(\Z)=1$. Consequently, $\si(W_0\times\set{\yz}\times\T)=1$ for all open neighborhoods $W_0$ of $\xz$. Since $\si$ is a regular measure, we conclude that $\si(\set{\xz}\times\set{\yz}\times\T)=1$. 
\end{proof}

We are now in a position to show that any representing measure for the functional $\d_{\p}$ is the Dirac measure concentrated at the point $\p\in\MD\times\T$.

\begin{lem}\label{lem3.3}
If $\pz\in\bta$, then the Dirac measure concentrated at $\pz$ is the unique representing measure for $\d_{\pz}$.
\end{lem}

\begin{proof}
Let $\pz=(\xz,\yz,\nuz)\in\bta$ and let $\si$ be a representing measure for $\d_{\pz}$. It then follows from Lemma~\ref{lem3.2} that $\si(\set{\xz}\times\set{\yz}\times\T)=\si(\Z)=1$. We shall prove that $\si(\set{\xz}\times\set{\yz}\times\set{\nuz})=1$. By virtue of condition \eqref{(7c)}, there exists $h_0\in\A$ such that $h_0(\xz)=0$ and $d(h_0)(\yz)=1$. Therefore, we conclude that $\TH_0(\xz,\yz,\nu)=\nu$ for all $\nu\in\T$ by \eqref{tilde}. We deduce that
\[
\nuz
=\d_{\pz}(\TH_0)
=\int_{\Z}\TH_0\,d\si
=\int_{\set{\xz}\times\set{\yz}\times\T}\TH_0\,d\si
=\int_{\set{\xz}\times\set{\yz}\times\T}\nu\,d\si(\nu).
\]
This implies that $\D\int_{\set{\xz}\times\set{\yz}\times\T}(\nuz-\nu)\,d\si(\nu)=0$, and therefore $\D\int_{\set{\xz}\times\set{\yz}\times\T}(1-\ov{\nuz}\nu)\,d\si(\nu)=0$. Because $\si$ is a probability measure, we have that
\[
\int_{\set{\xz}\times\set{\yz}\times\T}(1-\Re(\ov{\nuz}\nu))\,d\si(\nu)
=\Re\int_{\set{\xz}\times\set{\yz}\times\T}(1-\ov{\nuz}\nu)\,d\si(\nu)=0.
\]
Since $1-\Re(\ov{\nuz}\nu)>0$ for all $\nu\in\T\setminus\set{\nuz}$, the last equations imply that $\si(\set{\xz}\times\set{\yz}\times(\T\setminus\set{\nuz}))=0$, which yields $\si(\set{\xz}\times\set{\yz}\times\set{\nuz})=\si(\set{\xz}\times\set{\yz}\times\T)=1$. We thus conclude that $\si$ is the Dirac measure concentrated at the point $\pz=(\xz,\yz,\nuz)$.
\end{proof}

We shall now prove a property of the Choquet boundary $\chta$ of the Banach space $\TA$. This property is deeply connected with the set $\ext$.

\begin{lem}\label{lem3.4}
The set $\MD\times\T$ is a subset of $\chta$. In particular, $\chta$ is a dense subset of $\Z$.
\end{lem}

\begin{proof}
Let us fix $\p\in\bta$ arbitrarily
and prove that $\d_{\p} \in \ext$.
Let $\xio,\xit\in\TA_1^*$ be such that
$\d_{\p}=(\xio+\xit)/2$. 
By equation \eqref{tilde} and condition \eqref{(4)},
we have $\xio(\wt{\unit})+\xit(\wt{\unit})=2\d_{\p}(\wt{\unit})=2$.
As $\xi_j\in\TA_1^*$, we obtain $|\xi_j(\wt{\unit})|\leq1$.
Consequently, $\xi_j(\wt{\unit})=1=\V{\xi_j}$ for $j=1,2$.
Let $\si_j$ be a representing measure for $\xi_j$.
That is, $\D\xi_j(\TF)=\int_{\Z}\TF\,d\si_j$ for $\TF\in\TA$.
Therefore, we conclude that $\si_j$ is a probability measure
(see \cite[p. 81]{bro}).
Since $(\si_1+\si_2)/2$ is also a representing measure for $\d_{\p}$,
we derive from Lemma~\ref{lem3.3}
that $(\si_1+\si_2)/2$ is the Dirac measure, $\tau_{\p}$,
concentrated at $\p$. 
Since $\si_j$ is a positive measure, it follows that
$\si_j(G)=0$ for each
Borel set $G$ that does not contain the point $p$.
Therefore, we conclude that $\si_j=\tau_{\p}$ for $j=1,2$.
Consequently, we have $\xio=\d_{\p}=\xit$.
It follows that $\d_{\p}$ is an extreme point of $\TA_1^*$,
and thus $\bta\subset\chta$ as asserted.
Having in mind that $\MD$ is a dense subset of
$\PA\times\PB$
with $\bta\subset\chta\subset\Z$,
we conclude that $\chta$ is dense in $\Z$.
\end{proof}

We define a topological subspace $\del$ of $\Aos$ with the weak*-topology as
\[
\del=\set{\la\d_{\p}:\la\in\T,\,\p\in\Z}.
\]
The Arens--Kelley theorem states that
\begin{equation}\label{AK}
\ext=\set{\la\d_{\p}:\la\in\T,\ \p\in\chta}
\end{equation}
(see \cite[Corollary~2.3.6]{fle1} and \cite[Corollary~2.18]{hat1}).
By applying Lemma~\ref{lem3.4} to the last equation,
we obtain the following result:
\[
\set{\la\d_{\p}:\la\in\T, \ \p\in\bta}\subset\ext\subset\del.
\]
Define a mapping $\h\colon\T\times\Z\to\del$ as
\begin{equation}\label{h}
\h(\la,\p)=\la\d_{\p}
\qq((\la,\p)\in\T\times\Z).
\end{equation}

The following separation property for $\Z$ is a direct consequence of condition \eqref{(7)} and equation \eqref{tilde}, which will be utilized
in proving the next lemma.

\begin{prop}\label{prop3.5}
Let $\p_k=(\x_k,\y_k,\nu_k)\in\Z$ for $k=0,1,2$.
\begin{enumerate}
\item If $\xz\not\in\set{\xo,\xt}$, then there exists $f_0\in\A$ such that $\TF_0(\pz)=1$ and $\TF_0(\po)=0=\TF_0(\pt)$.
\item If $\yz\not\in\set{\yo,\yt}$, then there exists $f_1\in\A$ such that $\TF_1(\pz)=\nuz$ and $\TF_1(\po)=0=\TF_1(\pt)$.
\end{enumerate}
In particular, if $\pz\neq\po$, then there exists $f_2\in\A$ such that $\TF_2(\pz)\neq\TF_2(\po)$.
\end{prop}

\begin{proof}
If $\xz\not\in\set{\xo,\xt}$, then by virtue of condition \eqref{(7a)}, there exists $f_0\in\A$ such that $f_0(\xz)=1$, $f_0(\x_j)=0$ for $j=1,2$, and $d(f_0)(\y_k)=0$ for $k=0,1,2$. By combining these equations with \eqref{tilde}, we conclude that $\TF_0(\pz)=1$ and $\TF_0(\po)=0=\TF_0(\pt)$.

Now, we consider the case in which $\yz\not\in\set{\yo,\yt}$. It follows from \eqref{(7b)} that there exists $f_1\in\A$ such that $f_1(\x_k)=0$ for $k=0,1,2$, $d(f_1)(\yz)=1$ and $d(f_1)(\y_j)=0$ for $j=1,2$. It follows that $\TF_1(\pz)=\nuz$ and $\TF_1(\po)=0=\TF_1(\pt)$
by \eqref{tilde} with the preceding equations.

Finally, we shall assume that $\pz\neq\po$. We shall prove that $\TF_2(\pz)\neq\TF_2(\po)$ for some $f_2\in\A$. It is necessary to consider the case in which $\xz=\xo$, $\yz=\yo$ and $\nuz\neq\nuo$ by the fact proven above. By condition \eqref{(7c)}, there exists $f_2\in\A$ such that $f_2(\xz)=0$ and $d(f_2)(\yz)=1$. We infer from equation \eqref{tilde} that $\TF_2(\pz)=\nuz\neq\nuo=\TF_2(\po)$, since $\xz=\xo$ and $\yz=\yo$.
\end{proof}

We will prove that the map $\h$ defined as in equation \eqref{h} is a homeomorphism from $\T\times\Z$ onto $\del$. This property is of significant importance in the characterization of surjective isometries on $\TA$.

\begin{lem}\label{lem3.7}
The map $\h\colon\T\times\Z\to\del$ is a homeomorphism that satisfies $\h(\T\times\chta)=\ext$. In particular, the set $\del$ is a compact subset of the set $\Aos$.
\end{lem}

\begin{proof}
Clearly, $\h$ is surjective. We shall show that $\h$ is injective. Let us assume that $\h(\laz,\pz)=\h(\lao,\po)$ for $\laz,\lao\in\T$ and $\pz,\po\in\Z$. Then we have that $\laz\d_{\pz}=\lao\d_{\po}$. 
Note that $\laz=\laz\d_{\pz}(\wt{\unit})=\lao\d_{\po}(\wt{\unit})=\lao$ by using \eqref{tilde}. 
This implies that $\d_{\pz}=\d_{\po}$. It follows that $\pz=\po$, since $\TA$ separates the points of $\Z$ by Proposition~\ref{prop3.5}. Consequently, $(\laz,\pz)=(\lao,\po)$, as required. 

We can verify that $\h$ is a continuous map from $\T\times\Z$ with the product topology to $\del$ with the relative weak*-topology. Therefore, $\h$ is a bijective continuous map from the compact space $\T\times\Z$ onto the Hausdorff space $\del$. Consequently, $\h$ is a homeomorphism, as claimed. We obtain $\h(\T\times\chta)=\ext$ by
equation \eqref{AK}.
\end{proof}

\section{An induced isometry on $\TA$}

Let $T$ be a surjective, possibly nonlinear, isometry from $(\A,\VA{\cdot})$ onto itself. In this section, we shall introduce a surjective isometry from $(\TA,\Vinf{\cdot})$ onto itself by using $T$. We define a mapping $T_0$ on $(\A,\VA{\cdot})$ as follows:
\begin{equation}\label{T_0}
T_0 = T - T(0).
\end{equation}
By invoking the Mazur--Ulam theorem \cite{maz,nic}, we conclude that $T_0$ is a surjective, {\it real linear} isometry from $(\A,\VA{\cdot})$ onto itself. 

Let $S$ be the mapping defined by $UT_0U^{-1}$ (see \eqref{U}). Since $U$ is a surjective complex linear isometry, the mapping $S$ from $(\TA,\Vinf{\cdot})$ onto itself is a well-defined surjective real linear isometry.
\[
\begin{CD}
A@>{T_0}>>A\\
@V{U}VV
@VV{U}V\\
\TA@>>{S}>\TA
\end{CD}
\]
Hence $SU=UT_0$, and this can be expressed as
\begin{equation}\label{S}
S(\TF) = \widetilde{T_0(f)} \qq (f\in\A).
\end{equation}
We define a mapping $S_*\colon\TA^*\to\TA^*$ by
\begin{equation}\label{S_*}
S_*(\xi)(\TF) = \Re[\xi(S(\TF))] - i \, \Re [\xi(S(i\TF))]
\end{equation}
for $\xi\in\TA^*$ and $\TF\in\TA$.
The mapping $S_*$ is a surjective real linear isometry with respect to the operator norm on $\TA^*$. For a detailed proof, see \cite[Proposition~5.17]{rud}. This shows that $S_*(\ext) = \ext$.

The following lemma will demonstrate that the mapping $S_*$ preserves the set $\del$. This property is significant in characterizing the surjective real linear isometry $S$ from $(\TA,\Vinf{\cdot})$ onto itself.

\begin{lem}\label{lem3.8}
$S_*(\del)=\del$.
\end{lem}

\begin{proof}
Let $\h\colon\T\times\Z\to\del$ be the map defined as in equation \eqref{h}. As shown in Lemma~\ref{lem3.7}, $\h$ is a homeomorphism that satisfies the following equation:
\[
\h(\T\times\chta)=\ext= S_*(\ext).
\]
By Lemma~\ref{lem3.4}, we have $\bta\subset\chta\subset\Z$. Therefore, the last equations imply that
\[
S_*(\h(\T\times\bta))
\subset S_*(\h(\T\times\chta))
=S_*(\ext)
=\ext
=\h(\T\times\chta)
\subset\h(\T\times\Z).
\]
By the definition of the map $\h$, we obtain $\h(\T\times\Z)=\set{\la\d_{\p}:\la\in\T,\,\p\in\Z}=\del$. Consequently, we conclude that
\[
S_*(\h(\T\times\bta))\subset
\h(\T\times\Z)=\del.
\]
Let $\clk{M}{L}$ denote the closure of a subset $L$ in a topological space $M$. By definition, $\MD$ is dense in $\PA\times\PB$, and thus $\clk{\Z}{\bta}=\PA\times\PB\times\T=\Z$. Since the map $\h\colon\T\times\Z\to\del$ is a homeomorphism, we obtain $\h(\clk{\T\times\Z}{L})=\clk{\del}{\h(L)}$ for any $L\subset\T\times\Z$. It follows that
\[
\del=\h(\T\times\Z)
=\h(\T\times\clk{\Z}{\bta})
=\h(\clk{\T\times\Z}{\T\times\MD\times\T})
=\clk{\del}{\h(\T\times\bta)}.
\]
From the preceding equations, we infer that
\[
S_*(\del)=S_*(\clk{\del}{\h(\T\times\bta)}).
\]
We see that $S_*\colon\TA^*\to\TA^*$ is bijective, since it is a surjective isometry. By virtue of the definition of $S_*$, we can show that this map is continuous with respect to the weak*-topology
on $\TA^*$. Similarly, we can prove that the map $S_*^{-1}$ is continuous with respect to the weak*-topology on $\TA^*$. Therefore, we conclude that $S_*\colon\TA^*\to\TA^*$ is a homeomorphism with respect to the weak*-topology. Since $S_*(\h(\T\times\bta))\subset\del$, we obtain
\[
S_*(\del)
=S_*(\clk{\del}{\h(\T\times\bta)})
=\clk{\TA^*}{S_*(\h(\T\times\bta))}
\subset\clk{\TA^*}{\del}
=\del,
\]
where we have used the fact that $\del$ is compact by Lemma~\ref{lem3.7}. This consequently leads to the conclusion that $S_*(\del)\subset\del$. The same arguments, when applied to $S_*^{-1}$ in place of $S_*$, yields the result that $S_*^{-1}(\del)\subset\del$, which in turn implies that $S_*(\del)=\del$.
\end{proof}

\begin{defn}\label{defn1}
Let $\h\colon\T\times\Z\to\del$ be the homeomorphism
as defined in 
\eqref{h} (see Lemma~\ref{lem3.7}).
Let $\pio\colon\T\times\Z\to\T$
and $\pit\colon\T\times\Z\to\Z$ be the natural projections
from $\T\times\Z$ to the first and second coordinate,
respectively.
Two maps, $\al\colon\T\times\Z\to\T$
and $\Phi\colon\T\times\Z\to\Z$, are defined as follows:
$\al=\pio\circ \h^{-1}\circ S_*\circ\h$
and $\Phi=\pit\circ\h^{-1}\circ S_*\circ\h$.
Since $S_*(\del)=\del$ as proven in Lemma~\ref{lem3.8},
it follows that the maps $\al$ and $\Phi$ are well-defined.
\end{defn}
\[
\begin{CD}
\T\times\Z@>{}>>\T\times\Z\\
@V{\h}VV
@VV{\h}V\\
\del@>>{S_*}>\del
\end{CD}
\]
By the definitions of maps $\al$ and $\Phi$, we have that
$(\h^{-1}\circ S_*\circ\h)(\la,\p)=(\al(\la,\p),\Phi(\la,\p))$
for all $(\la,\p)\in\T\times\Z$.
Therefore, we conclude that
$(S_*\circ\h)(\la,\p)=\h(\al(\la,\p),\Phi(\la,\p))$,
which in turn implies that
$S_*(\la\d_{\p})=\al(\la,\p)\d_{\Phi(\la,\p)}$.
For the sake of simplicity in notation,
we shall henceforth write $\al(\la,\p)=\al_\la(\p)$.
Subsequently, we have
\begin{equation}\label{Sstar}
S_*(\la\d_{\p})=\al_\la(\p)\dphi{\la,\p}
\end{equation}
for all $(\la,\p)\in\T\times\Z$.
It can be observed that both $\al$ and $\Phi$ are surjective
continuous maps, since $\h$ and $S_*$ are homeomorphisms.

The remainder of this section will be dedicated to an examination of the maps $\al$ and $\Phi$. The following lemma will demonstrate the connection between two maps, namely $\alo$ and $\al_i$.

\begin{lem}\label{lem4.1}
For each $\p\in\Z$, we obtain $\al_i(\p)=i\al_1(\p)$ or $\al_i(\p)=-i\al_1(\p)$.
\end{lem}

\begin{proof}
Let us fix an arbitrary $\p\in\Z$. We then set $\la_0=(1+i)/\sqrt{2}\in\T$. By the real linearity of $S_*$, we obtain the following:
\[
\sqrt{2}\,\al_{\la_0}(\p)\dphi{\la_0,\p}
=S_*(\sqrt{2}\,\la_0\d_{\p})
=S_*(\d_{\p}) + S_*(i\d_{\p})
=\al_1(\p)\dphi{1,\p}+\al_i(\p)\dphi{i,\p}.
\]
The last equations, evaluated at $\wt{\unit}\in\TA$,
yield the result that
$\sqrt{2}\,\al_{\la_0}(\p)=\al_1(\p)+\al_i(\p)$.
Since $|\al_\la(\p)|=1$ for all $\la\in\T$,
we have that
$\sqrt{2}=|\al_1(\p)+\al_i(\p)|=|1+\al_i(\p)\ov{\al_1(\p)}|$.
It can be seen that $\al_i(\p)\ov{\al_1(\p)}$ is
equal to $i$ or $-i$,
since $|\al_i(\p)\ov{\al_1(\p)}|=1$.
Therefore, we conclude that $\al_i(\p)=i\al_1(\p)$
or $\al_i(\p)=-i\al_1(\p)$.
\end{proof}

In the following lemma, we will examine the connection between two functionals $S_*(i\d_{\p})$ and $S_*(\d_{\p})$ for each $\p\in\Z$.

\begin{lem}\label{lem4.2}
There is a continuous map $\ez\colon\Z\to\set{\pm 1}$ such that $S_*(i\d_{\p})=i\ez(\p)\alo(\p)\dphi{i,\p}$ for every $\p\in\Z$. Moreover, for each $\xz\in\PA$ and $\yz\in\PB$, the map $\ez(\xz,\yz,\cdot)\colon\T\to\set{\pm1}$, which maps $\nu\in\T$ to $\ez(\xz,\yz,\nu)\in\set{\pm1}$, is a constant function.
\end{lem}

\begin{proof}
We define the sets $\Z_+=\set{\p\in\Z:\al_i(\p)= i\alo(\p)}$ and $\Z_-=\set{\p\in\Z:\al_i(\p)=-i\alo(\p)}$. We see that $\Z=\Z_+\cup\Z_-$ and $\Z_+\cap\Z_-=\emptyset$, since $\al_i(\p)=i\alo(\p)$ or $\al_i(\p)=-i\alo(\p)$ for each $\p\in\Z$ as shown in Lemma~\ref{lem4.1}. By virtue of the continuity of the function $\al\colon\T\times\Z\to\T$, the function $\al_\la\colon\Z\to\T$, which maps each point $\p\in\Z$ to $\al(\la,\p)$ for a fixed $\la\in\T$, is continuous on $\Z$. It can be observed that both $\Z_+$ and $\Z_-$ are closed subsets of $\Z$. The function $\ez\colon\Z\to\set{\pm 1}$, defined as
\[
\ez(\p)=
\begin{cases}
\ph{-}1&\p\in\Z_+\\[1mm]
-1&\p\in\Z_-
\end{cases},
\]
is therefore continuous on $\Z$. Since $\Z=\Z_+\cup\Z_-$, we obtain $\al_i(\p)=i\ez(\p)\alo(\p)$ for every $\p\in\Z$. By combining the last equation with equation \eqref{Sstar}, we conclude that $S_*(i\d_{\p})=\al_i(\p)\d_{\Phi(i,\p)}=i\ez(\p)\alo(\p)\d_{\Phi(i,\p)}$ for all $\p\in\Z$.

Finally, fix $\xz\in\PA$ and $\yz\in\PB$ arbitrarily. It is evident that the range of $\T$, $\ez(\xz,\yz,\T)$, under the map $\ez(\xz,\yz,\cdot)$ is a connected subset of $\set{\pm1}$, since $\ez$ is continuous on $\Z$. Therefore, $\ez(\xz,\yz,\cdot)$ is a constant function on $\T$.
\end{proof}

The following lemma establishes the connection between three points, namely, $\Phi(\la,\p)$, $\Phi(1,\p)$, and $\Phi(i,\p)$ in $\Z$ for each point $\p\in\Z$. This is a crucial aspect examining
the mapping $\Phi$.

\begin{lem}\label{lem4.3}
Let $\ez\colon\Z\to\set{\pm 1}$ be the continuous function defined as in Lemma~\ref{lem4.2}, and let $\la=a+ib\in\T$ with $a,b\in\R$ and $\p\in\Z$. Then we have that 
\begin{equation}\label{lem4.3.1}
\la^{\ez(\p)}\TF(\Phi(\la,\p))
=a\TF(\Phi(1,\p))+ib\ez(\p)\TF(\Phi(i,\p))
\end{equation}
for all $\TF\in\TA$.
\end{lem}

\begin{proof}
It should be recalled that, according to Lemma~\ref{lem4.2},
$S_*(i\d_{\p})=i\ez(\p)\alo(\p)\dphi{i,\p}$.
Since $S_*$ is real linear, we deduce from \eqref{Sstar} that
$$
\al_\la(\p)\dphi{\la,\p}
=S_*(\la \d_{\p})
=aS_*(\d_{\p})+bS_*(i\d_{\p})
=a\alo(\p)\dphi{1,\p}+ib\ez(\p)\alo(\p)\dphi{i,\p},
$$
and thus 
\begin{equation}\label{lem4.3.2}
\al_\la(\p)\dphi{\la,\p}=\alo(\p)(a\dphi{1,\p}+ib\ez(\p)\dphi{i,\p}).
\end{equation}
The evaluation at $\tio\in\TA$ shows that $\al_\la(\p)=\alo(\p)(a+ib\ez(\p))=\alo(\p)\la^{\ez(\p)}$. This is due to the fact that $\la=a+ib\in\T$ and $\ez(\p)\in\set{\pm 1}$. 
Entering the previous equation into \eqref{lem4.3.2}, we obtain $\la^{\ez(\p)}\dphi{\la,\p}=a\dphi{1,\p}+ib\ez(\p)\dphi{i,\p}$, and thus $\la^{\ez(\p)}\TF(\Phi(\la,\p))=a\TF(\Phi(1,\p))+ib\ez(\p)\TF(\Phi(i,\p))$ for all $\TF\in\TA$.
\end{proof}

\begin{defn}\label{defn2}
Let $q_1\colon\Z\to\PA$, $q_2\colon\Z\to\PB$ and $q_3\colon\Z\to\T$ be the natural projections from $\Z=\PA\times\PB\times\T$ onto the $m$-th coordinate of $\Z$ for $m=1,2,3$. In accordance with Definition~\ref{defn1}, the map $\Phi\colon\T\times\Z\to\Z$ is defined in terms of three maps $\phi\colon\T\times\Z\to\PA$, $\psi\colon\T\times\Z\to\PB$ and $\varphi\colon\T\times\Z\to\T$ given by $\phi=q_1\circ\Phi$, $\psi=q_2\circ\Phi$ and $\varphi=q_3\circ\Phi$.
\end{defn}

In agreement with the definition of the three maps, $\phi$, $\psi$ and $\varphi$, we have that
\begin{equation}\label{Phi}
\Phi(\la,\p)=(\phi(\la,\p),\psi(\la,\p),\varphi(\la,\p))
\end{equation}
for every $(\la,\p)\in\T\times\Z$. It should be noted that the maps $\phi$, $\psi$ and $\varphi$ are surjective and continuous, as is the map $\Phi$ (see Definition~\ref{defn1}).

The structure of the map $\Phi$ will be elucidated by an examination of three maps: $\phi$, $\psi$ and $\varphi$. Firstly, it will be shown that the map $\phi$ is independent of the variable $\la\in\T$.

\begin{lem}\label{lem4.4}
The mapping $\phi(1,\cdot)$, which maps $\p\in\Z$ to $\phi(1,\p)\in\PA$, is a surjective continuous map that satisfies the equation $\phi(1,\p)=\phi(\la,\p)$ for all $\p\in\Z$ and $\la\in\T$.

For the sake of brevity, we shall henceforth write $\phi(1,\p)=\phi(\p)$ for all $\p\in\Z$.
\end{lem}

\begin{proof}
Let us fix an arbitrary $\p\in\Z$. We first prove that
$\phi(\la,\p)\in\set{\phi(1,\p),\phi(i,\p)}$
for all $\la\in\T$.
We may then suppose, on contrarily, that
there exists $\la_0\in\T\setminus\set{1,i}$
such that $\phi(\la_0,\p)\not\in\set{\phi(1,\p),\phi(i,\p)}$.
By virtue of Proposition~\ref{prop3.5},
there exists $f_0\in\A$ such that
$\TF_0(\Phi(\la_0,\p))=1$ and
$\TF_0(\Phi(1,\p))=0=\TF_0(\Phi(i,\p))$.
By substituting these three equations
into \eqref{lem4.3.1},
we obtain $\la_0^{\ez(\p)}=0$.
This is in contradiction with the assumption
that $\la_0\in\T$.
Therefore, we conclude that
$\phi(\la,\p)\in\set{\phi(1,\p),\phi(i,\p)}$
for all $\la\in\T$.

We next prove that $\phi(1,\p)=\phi(i,\p)$.
For the sake of argument, let us suppose
that $\phi(1,\p)\neq\phi(i,\p)$.
Let $\la_1\in\T$ be defined as $\la_1=(1+i)/\sqrt{2}$.
We obtain
$\phi(\la_1,\p)\in\set{\phi(1,\p),\phi(i,\p)}$
as was previously demonstrated.
We now consider the case where $\phi(\la_1,\p)=\phi(1,\p)$.
In this context, it follows that
$\phi(i,\p)\not\in\set{\phi(1,\p)}
=\set{\phi(1,\p),\phi(\la_1,\p)}$,
since $\phi(1,\p)\neq\phi(i,\p)$.
By applying Proposition~\ref{prop3.5},
there exists $f_1\in\A$ that satisfies
$\TF_1(\Phi(i,\p))=1$ and
$\TF_1(\Phi(1,\p))=0=\TF_1(\Phi(\la_1,\p))$.
It follows from \eqref{lem4.3.1} that $0=i\ez(\p)$,
which is in contradiction with the assumption
that $\ez(\p)\in\set{\pm1}$.
A similar argument will yield a contradiction even if
we assume that $\phi(\la_1,\p)=\phi(i,\p)$.
We conclude that $\phi(1,\p)=\phi(i,\p)$
as was previously asserted.
Therefore, we conclude that the equation
$\phi(1,\p)=\phi(\la,\p)$ holds true
for all $\p\in\Z$ and $\la\in\T$.
This is due to the fact that
$\phi(\la,\p)\in\set{\phi(1,\p),\phi(i,\p)}$.

Finally, we prove that the map $\phi(1,\cdot)$ is surjective.
Let us take an arbitrary element $\x\in\PA$.
Since the map $\phi$ is surjective,
there exists $(\mu,q)\in\T\times\Z$
such that $\x=\phi(\mu,q)$.
In light of the preceding argument,
it follows that
$\phi(1,q)=\phi(\mu,q)=\x$,
which establishes that $\phi(1,\cdot)$ is surjective.
\end{proof}

A similar argument to that used in the proof of Lemma~\ref{lem4.4} shows that $\psi(1,\p)=\psi(\la,\p)$ for all $\p\in\Z$ and $\la\in\T$. We will therefore omit its proof.

\begin{lem}\label{lem4.5}
The mapping $\psi(1,\cdot)$, which maps $\p\in\Z$ to $\psi(1,\p)\in\PB$, is a surjective continuous map that satisfies the equation $\psi(1,\p)=\psi(\la,\p)$ for all $\p\in\Z$ and $\la\in\T$.

We shall henceforth write $\psi(1,\p)=\psi(\p)$ for all $\p\in\Z$.
\end{lem}

The following lemma establishes a connection between the two values of $\varphi(i,\p)$ and $\varphi(1,\p)$ for each $\p\in\Z$.

\begin{lem}\label{lem4.6}
There exists a continuous function $\eo\colon\Z\to\set{\pm1}$ such that $\varphi(i,\p)=\ez(\p)\eo(\p)\varphi(1,\p)$ for all $\p\in\Z$. Moreover, for each pair of elements, $\xz\in\PA$ and $\yz\in\PB$, the map $\eo(\xz,\yz,\cdot)\colon\T\to\set{\pm1}$, which maps $\nu\in\T$ to $\eo(\xz,\yz,\nu)\in\set{\pm1}$, is a constant function.
\end{lem}

\begin{proof}
Let us fix an arbitrary $\p\in\Z$.
We can choose $f_0\in\A$ in such a way that
$f_0(\phi(\p))=0$
and $d(f_0)(\psi(\p))=1$
by virtue of \eqref{(7c)}.
By setting $\laz=(1+i)/\sqrt{2}\in\T$,
we obtain the result that
$\TF_0(\Phi(\mu,\p))=\varphi(\mu,\p)$
for $\mu=\la_0,1,i$,
where we have used the equations
\eqref{tilde} and \eqref{Phi}
with $\phi(\mu,\p)=\phi(\p)$ and
$\psi(\mu,\p)=\psi(\p)$.
The last equation is then applied
to equation \eqref{lem4.3.1} to yield
the following result:
\[
\la_0^{\ez(\p)}\varphi(\la_0,\p)
=\laz^{\ez(\p)}\TF_0(\Phi(\laz,\p))
=\fr{1}{\sqrt{2}}\,\varphi(1,\p)
+\fr{i\ez(\p)}{\sqrt{2}}\,\varphi(i,\p).
\]
We derive from the preceding equation that
\[
\sqrt{2}
=|\varphi(1,\p)+i\ez(\p)\varphi(i,\p)|
=|1+i\ez(\p)\varphi(i,\p)\ov{\varphi(1,\p)}|.
\]
This is due to the fact that
$\varphi(\la,\p)\in\T$ for each $\la\in\T$.
It follows that
$i\ez(\p)\varphi(i,\p)\ov{\varphi(1,\p)}$ is
$i$ or $-i$,
since $i\ez(\p)\varphi(i,\p)\ov{\varphi(1,\p)}\in\T$.
Therefore, we conclude that either
$\varphi(i,\p)=\ez(\p)\varphi(1,\p)$
or $\varphi(i,\p)=-\ez(\p)\varphi(1,\p)$.
By virtue of the continuity of the functions
$\ez$, $\varphi(1,\cdot)$ and $\varphi(i,\cdot)$,
there exists a continuous function
$\eo\colon\Z\to\set{\pm1}$ such that
$\varphi(i,\p)=\ez(\p)\eo(\p)\varphi(1,\p)$ for all $\p\in\Z$.

It can be observed that the map $\eo(\xz,\yz,\cdot)$ is a constant
function on the connected set $\T$
for each $\xz\in\PA$ and $\yz\in\PB$,
since the function $\eo$ is continuous on the set $\Z$.
\end{proof}

\begin{nota}\label{nota3}
For the sake of brevity, we shall henceforth write
$\varphi(1,\p)=\varphi(\p)$
for each $\p\in\Z$.
By virtue of Lemma~\ref{lem4.6},
the following equation is true:
$\varphi(i,\p)=\ez(\p)\eo(\p)\varphi(\p)$
for $\p\in\Z$.

We denote $a+ib\e$ by $\left\langle a+ib\right\rangle^{\e}$
for $a,b\in\R$ and $\e\in\set{\pm1}$.
It thus follows that
$\left\langle\la\mu\right\rangle^{\e}=\left\langle\la\right\rangle^{\e}\left\langle\mu\right\rangle^{\e}$
for all $\la,\mu\in\C$ and $\e\in\set{\pm1}$:
Furthermore, if $\la\in\T$, then we have that
$\left\langle\la\right\rangle^{\e}=\la^{\e}$ for $\e\in\set{\pm1}$.
For simplicity, we shall also write
$\alo(\p)=\al(\p)$ for $\p\in\Z$.
By equation \eqref{Sstar} and Lemma~\ref{lem4.2},
we obtain the following result:
\begin{equation}\label{nota3.1}
S_*(\d_{\p})=\al(\p)\d_{\Phi(1,\p)}
\qq\mbox{and}\qq
S_*(i\d_{\p})=i\ez(\p)\al(\p)\d_{\Phi(i,\p)}
\qq(\p\in\Z).
\end{equation}
\end{nota}

We are now in a position to describe the form of the surjective real linear isometry $S$ from $(\TA,\Vinf{\cdot})$ onto itself.

\begin{lem}\label{lem4.7}
For each $f\in\A$ and $\p=(\x,\y,\nu)\in\Z$, the following equation holds:
\begin{equation}\label{lem5.1.1}
T_0(f)(\x)+d(T_0(f))(\y)\nu
=\left\langle\al(\p)f(\phi(\p))\right\rangle^{\ez(\p)}
+\left\langle\al(\p)\varphi(\p)d(f)(\psi(\p))\right\rangle^{\eo(\p)}.
\end{equation}
\end{lem}

\begin{proof}
Let us fix arbitrary $f\in\A$ and $\p\in\Z$.
In accordance with the definition of $S_*$
(see \eqref{S_*}), it follows that
$\Re[S_*(\xi)(\TF)]=\Re[\xi(S(\TF))]$
for every $\xi\in\TA^*$.
By substituting $\xi=\d_{\p}$ and $\xi=i\d_{\p}$
into the previous equation,
we obtain the following two equations:
\begin{align*}
&\Re[S_*(\d_{\p})(\TF)]=\Re[S(\TF)(\p)],\\
&\Re[S_*(i\d_{\p})(\TF)]=-\Im[S(\TF)(\p)].
\end{align*}
It follows that
\[
S(\TF)(\p)=\Re[S_*(\d_{\p})(\TF)]-i\,\Re[S_*(i\d_{\p})(\TF)].
\]
On the one hand, by substituting \eqref{nota3.1} into the last equation, we obtain that
\begin{equation}\label{lem3.11.3}
S(\TF)(\p)=\Re[\al(\p)\TF(\Phi(1,\p))]
+i\,\Im[\ez(\p)\al(\p)\TF(\Phi(i,\p))].
\end{equation}
Lemmas from \ref{lem4.4} to \ref{lem4.6} imply that
\begin{align*}
&\Phi(1,\p)=(\phi(\p),\psi(\p),\varphi(\p)),\\
&\Phi(i,\p)=(\phi(\p),\psi(\p),\varphi(i,\p)),
\end{align*}
where $\varphi(1,\p)=\varphi(\p)$ and $\varphi(i,\p)=\ez(\p)\eo(\p)\varphi(\p)$. On the other hand, we deduce from \eqref{tilde} that
\begin{align*}
&\TF(\Phi(1,\p))
=f(\phi(\p))+d(f)(\psi(\p))\varphi(\p),\\
&\TF(\Phi(i,\p))
=f(\phi(\p))+d(f)(\psi(\p))\ez(\p)\eo(\p)\varphi(\p).
\end{align*}
The preceding two equations yield the following equations:
\begin{align*}
&\Re[\al(\p)\TF(\Phi(1,\p))]
=\Re[\al(\p)f(\phi(\p))]+\Re[\al(\p)\varphi(\p)d(f)(\psi(\p))],\\
&\Im[\ez(\p)\al(\p)\TF(\Phi(i,\p))]
=\Im[\ez(\p)\al(\p)f(\phi(\p))]
+\Im[\eo(\p)\al(\p)\varphi(\p)d(f)(\psi(\p))].
\end{align*}
By substituting the two preceding equations into \eqref{lem3.11.3}, we deduce that
\begin{equation}\label{lem4.7.1}
S(\TF)(\p)
=\left\langle\al(\p)f(\phi(\p))\right\rangle^{\ez(\p)}
+\left\langle\al(\p)\varphi(\p)d(f)(\psi(\p))\right\rangle^{\eo(\p)}.
\end{equation}
By equation \eqref{S}, we have that $S(\TF)=\wt{T_0(f)}$.
By applying equation \eqref{tilde}, we can rewrite equation \eqref{lem4.7.1} as equation \eqref{lem5.1.1}.
\end{proof}

\section{Proof of the main results and examples}\label{sect5}

\subsection{Proof of Theorem 1}

The following lemma is relatively straightforward, yet it serves as a crucial foundation for the subsequent argument.

\begin{lem}\label{prop5.1}
If $c_0,c_1\in\C$ and $|c_0+\nu c_1| = 1$ for all $\nu\in\T$, then $|c_0|+|c_1|=1$ and $c_0c_1=0$.
\end{lem}

\begin{proof}
By hypothesis, it follows that $|c_0|+|c_1|\neq0$. Without loss of generality, we may and do assume that $c_0\neq 0$. We will demonstrate that $c_1=0$, which will then allow us to conclude that $|c_0|=1$.
Let us suppose, for the purpose of contradiction, that $c_1\neq0$. By taking $\nu_1=c_0c_1^{-1}|c_1|\,|c_0|^{-1}\in\T$, we obtain $|c_0+\nu_1c_1|=1=|c_0-\nu_1c_1|$. This is in accordance with our initial assumption. The preceding equations yield the following result:
\[
\left|c_0+\fr{c_0|c_1|}{|c_0|}\right|
=1
=\left|c_0-\fr{c_0|c_1|}{|c_0|}\right|,
\]
which implies that $|c_0|+|c_1|=||c_0|-|c_1||$. The last equation shows that if $|c_0|\geq|c_1|$, then $|c_1|=0$. Conversely, if $|c_0|<|c_1|$, then $|c_0|=0$. In any case, this leads to a contradiction,
and thus we may conclude that $c_1=0$. Therefore, by hypothesis, we have that $|c_0|=1$. Therefore, we conclude that $|c_0|+|c_1|=1$ and $c_0c_1=0$.
\end{proof}

Let $(\A,\VA{\cdot})$ be a Banach space with the properties from \eqref{(1)} through \eqref{(7)}, and let $T\colon\A\to\A$ be a surjective isometry. We have divided the proof of Theorem 1 into a series of claims.

\begin{claim}\label{cl1}
There exists $c\in\T$ such that $T_0(\unit)(\x)=c$ for all $\x\in\PA$ and $\al(\p)^{\ez(\p)}=c$ for all $\p\in\Z$.
\end{claim}

We shall first prove that there exists $x_0\in\PA$ such that $T_0(\unit)(\x_0)\neq 0$. Suppose, on the contrary, that $T_0(\unit)(\x)=0$ for all $\x\in\PA$. Hence $T_0(\unit)=0$ since $\cha$ is a boundary for $\A$. It follows that $d(T_0(\unit))=0$ since $d$ is a linear map, and thus $d(T_0(\unit))(\y)=0$ for all $\y\in\PB$. Substituting the two equations into equation \eqref{lem5.1.1} yields that $0=\al(\p)^{\ez(\p)}$ for all $\p\in\Z$, where we have used the fact that $d(\unit)=0$ by condition \eqref{(4)}. This contradicts that $\al(\p)\in\T$. Therefore, we conclude that $T_0(\unit)(\x_0)\neq0$ for some $\x_0\in\PA$.

Fixing an arbitrary $\y\in\PB$, we infer from equation \eqref{lem5.1.1} that for all $\nu\in\T$, we have that 
\[
|T_0(\unit)(\x_0)+d(T_0(\unit))(\y)\nu|=|\al(\pz)^{\ez(\pz)}|=1,
\]
where $\pz=(\xz,\y,\nu)\in\Z$. In light of the fact that $T_0(\unit)(\xz)\neq0$, Lemma \ref{prop5.1} shows that $d(T_0(\unit))(\y)=0$ and that $|T_0(\unit)(\xz)|=1$. Since $\y\in\PB$ was chosen arbitrarily, we deduce that $d(T_0(\unit))=0$ on $\Y$. Therefore, $T_0(\unit)$ is a constant function on $\X$ in accordance with \eqref{(4)}. By setting $c=T_0(\unit)$, we have that $c\in\T$. By applying equation \eqref{lem5.1.1}, we conclude that $c=T_0(\unit)(\x)=\al(\p)^{\ez(\p)}$ for all $\p=(\x,\y,\nu)\in\Z$.
\qed

\begin{claim}\label{lem5.3}
Given $\x\in\PA$ and $\y\in\PB$, we have $\phi(\x,\y,\nuo)=\phi(\x,\y,\nut)$ for all $\nuo,\nut\in\T$.
\end{claim}

For each $\nu\in\T$, we set $\p_{\nu}=(\x,\y,\nu)$. Let $\nuo,\nut\in\T$ with $\nuo\neq\nut$. Firstly, we claim $\phi(\p_\nu)\in\set{\phi(\p_{\nuo}),\phi(\p_{\nut})}$ for all $\nu\in\T$. Let us assume, for the sake of contradiction, that there exists a point $\nuz\in\T$ such that $\phi(\p_{\nuz})\not\in\set{\phi(\p_{\nuo}),\phi(\p_{\nut})}$. By virtue of condition \eqref{(7a)}, there exists $f_0\in\A$ such that $f_0(\phio(\p_{\nuz}))=1$, $f_0(\phio(\p_{\nu_j}))=0$ for $j=1,2$ and $d(f_0)(\psio(\p_{\nu_k}))=0$ for $k=0,1,2$. By substituting the previous equations into \eqref{lem5.1.1}, we obtain the following two equations:
\begin{align*}
&T_0(f_0)(\x)+d(T_0(f_0))(\y)\nuz
=\al(\p_{\nuz})^{\ez(\p_{\nuz})},\\
&T_0(f_0)(\x)+d(T_0(f_0))(\y)\nu_j
=0
\qq(j=1,2).
\end{align*}
Since $\nuo\neq\nut$, we deduce that $d(T_0(f_0))(\y)=0=T_0(f_0)(\x)$. This yields that $0=\al(\p_{\nuz})^{\ez(\p_{\nuz})}$. This is in contradiction with the assumption that $\al(\p_{\nuz})\in\T$, and this proves our claim.

In view of the fact that $\phi$ is continuous on $\Z$ by Lemma \ref{lem4.4}, we conclude that the set $\set{\phi(\p_\nu):\nu\in\T}$ is connected. It follows that $\phi(\p_{\nuo})=\phi(\p_{\nut})$ by our claim.
\qed

\begin{claim}\label{lem5.4}
Given $\x\in\PA$ and $\y\in\PB$, we have $\psi(\x,\y,\nuo)=\psi(\x,\y,\nut)$ for all $\nuo,\nut\in\T$.
\end{claim}

It is proved with a similar argument to that used to prove Claim \ref{lem5.3}.
\qed

\begin{nota} Let $\p=(\x,\y,\nu)\in\Z$. We may write $\phi(\p)=\phi(\x,\y)$ and $\psi(\p)=\psi(\x,\y)$ by Claims \ref{lem5.3} and \ref{lem5.4}. In accordance with Lemmas~\ref{lem4.2} and \ref{lem4.6}, we may express $\ez(\p)=\ez(\x,\y)$ and $\eo(\p)=\eo(\x,\y)$. By setting $\et=\ez\eo$, we obtain $\al(\p)^{\eo(\p)}=c^{\et(\x,\y)}$ since $\al(\p)^{\ez(\p)}=c$ by Claim \ref{cl1}. 
\end{nota}

\begin{claim}\label{lem5.5}
For each $f\in\A$, $\x\in\PA$ and $\y\in\PB$, we have the
following equations:
\begin{equation}\label{lem5.5.1}
T_0(f)(\x)+d(T_0(f))(\y)\nu
=c\left\langle f(\phi(\x,\y))\right\rangle^{\ez(\x,\y)}
+c^{\et(\x,\y)}\left\langle\varphi(\p)d(f)(\psi(\x,\y))\right\rangle^{\eo(\x,\y)}
\end{equation}
and 
\begin{align}\label{lem5.5.2}
\begin{aligned}
&T_0(f)(\x)=c\left\langle f(\phi(\x,\y))\right\rangle^{\ez(\x,\y)},\\
&d(T_0(f))(\y)
=c^{\et(\x,\y)}\left\langle\varphio(\x,\y)d(f)(\psi(\x,\y))\right\rangle^{\eo(\x,\y)},
\end{aligned}
\end{align}
where we set $\varphio(\x,\y)=\varphi(\x,\y,1)$.
\end{claim}

Let $\x\in\PA$ and $\y\in\PB$. In terms of the notation above, we can restate equation \eqref{lem5.1.1} as equation \eqref{lem5.5.1}. For each $\nu\in\T$, we set $\p_{\nu}=(\x,\y,\nu)$. By 
condition~\eqref{(7c)}, there exists $f_0\in\A$ such that $f_0(\phi(\x,\y))=0$ and $d(f_0)(\psi(\x,\y))=1$. By substituting these two equations into \eqref{lem5.5.1}, we obtain
\begin{equation}\label{lem5.5.3}
T_0(f_0)(\x)+d(T_0(f_0))(\y)\nu
=c^{\et(\x,\y)}\varphi(\p_{\nu})^{\eo(\x,\y)}
\end{equation}
for all $\nu\in\T$. Upon taking the modulus of both sides of the last equation, we obtain that
\[
|T_0(f_0)(\x)+d(T_0(f_0))(\y)\nu|=1
\]
for all $\nu\in\T$. In view of Lemma \ref{prop5.1}, we infer that
\begin{align}\label{lem5.5.4}
\begin{aligned}
&|T_0(f_0)(\x)|+|d(T_0(f_0))(\y)|=1,\\
&T_0(f_0)(\x)\cdot d(T_0(f_0))(\y)=0.
\end{aligned}
\end{align}
Our next objective is to demonstrate that $d(T_0(f_0))(\y)\neq0$. Assume otherwise that $d(T_0(f_0))(\y)=0$. We derive from \eqref{lem5.5.3} that $T_0(f_0)(\x)
=c^{\et(\x,\y)}\varphi(\p_{\nu})^{\eo(\x,\y)}$. By combining the last equation with \eqref{lem5.5.1}, we obtain that
\[
T_0(f)(\x)+d(T_0(f))(\y)\nu=c\left\langle f(\phi(\x,\y))\right\rangle^{\ez(\x,\y)}+T_0(f_0)(\x)\left\langle d(f)(\psi(\x,\y))\right\rangle^{\eo(\x,\y)}
\]
for all $f\in\A$ and $\nu\in\T$. By applying \eqref{(7c)}, there exists $f_1\in\A$ such that $d(T_0(f_1))(\y)=1$. This is a consequence of the fact that $T_0$ is surjective. Upon substituting $f=f_1$ into the aforementioned equation, we have that
\[
T_0(f_1)(\x)+\nu=c\left\langle f_1(\phi(\x,\y))\right\rangle^{\ez(\x,\y)}+T_0(f_0)(\x)\left\langle d(f_1)(\psi(\x,\y))\right\rangle^{\eo(\x,\y)}
\]
for all $\nu\in\T$. This is impossible since the right-hand side of the last equation is independent of $\nu\in\T$. Therefore, we conclude that $d(T_0(f_0))(\y)\neq0$. 

Consequently, this implies that $T_0(f_0)(\x)=0$ and that $|d(T_0(f_0))(\y)|=1$ in accordance with equation \eqref{lem5.5.4}. By combining the preceding equation with equation \eqref{lem5.5.3}, we obtain the following:
\begin{equation}\label{lem5.5.5}
d(T_0(f_0))(\y)\nu
=c^{\et(\x,\y)}\varphi(\p_\nu)^{\eo(\x,\y)}.
\end{equation}
Substituting the last equation into \eqref{lem5.5.1} yields that
\[
T_0(f)(\x)+d(T_0(f))(\y)\nu =c\left\langle f(\phi(\x,\y))\right\rangle^{\ez(\x,\y)} +d(T_0(f_0))(\y)\nu\left\langle d(f)(\psi(\x,\y))\right\rangle^{\eo(\x,\y)}
\]
for all $f\in\A$ and $\nu\in\T$. A comparison of the constant and linear terms of the last equation with respect to $\nu$ reveals that
\begin{align}
&T_0(f)(\x)=c\left\langle f(\phi(\x,\y))\right\rangle^{\ez(\x,\y)},
\nonumber\\
&d(T_0(f))(\y)
=d(T_0(f_0))(\y)\left\langle d(f)(\psi(\x,\y))\right\rangle^{\eo(\x,\y)}
\label{lem5.5.6}
\end{align}
for all $f\in\A$.
For the sake of simplicity, we shall rewrite $\varphi(\x,\y,1)=\varphio(\x,\y)$.
The equation \eqref{lem5.5.5} is rewritten as 
$$
d(T_0(f_0))(\y)=c^{\et(\x,\y)}\varphio(\x,\y)^{\eo(\x,\y)}.
$$
Subsequently, we can reformulate equation \eqref{lem5.5.6} as 
\[
d(T_0(f))(\y)=c^{\et(\x,\y)}\left\langle\varphio(\x,\y)d(f)(\psi(\x,\y))\right\rangle^{\eo(\x,\y)}
\]
for all $f\in\A$.
\qed

\begin{claim}\label{lem5.6}
For each $\xz\in\PA$, the maps $\phi(\xz,\cdot)\colon\PB\to\PA$ and $\ez(\xz,\cdot)\colon\PB\to\set{\pm1}$ which map $\y\in\PB$ to $\phi(\xz,\y)$ and $\ez(\xz,\y)$, respectively, 
are both constant.

For the sake of brevity, we shall henceforth write $\phi(\x,\y)=\phi(\x)$ and $\ez(\x,\y)=\ez(\x)$ for $\x\in\PA$ and $\y\in\PB$. This yields the following equation:
\begin{equation}\label{lem5.6.1}
T_0(f)(\x)=c\left\langle f(\phi(\x))\right\rangle^{\ez(\x)}
\end{equation}
for all $f\in\A$ and $\x\in\PA$.
Moreover, the maps $\phi$ and $\ez$, defined on $\PA$,
are continuous on $\PA$.
\end{claim}

Let us fix an arbitrary $\xz\in\PA$. We shall show that $\phi(\xz,\yo)=\phi(\xz,\yt)$ for all $\yo,\yt\in\PB$. We may, however, suppose that there exist $\yo,\yt\in\PB$ such that
$\phi(\xz,\yo)\neq\phi(\xz,\yt)$. By
condition~\eqref{(7a)}, there exists $f_0\in\A$ such that $f_0(\phi(\xz,\yo))=1$ and $f_0(\phi(\xz,\yt))=0$. We infer from \eqref{lem5.5.2} that
\[
c=
c\left\langle f_0(\phi(\xz,\yo))\right\rangle^{\ez(\xz,\yo)}
=T_0(f_0)(\xz)
=c\left\langle f_0(\phi(\xz,\yt))\right\rangle^{\ez(\xz,\yt)}=0,
\]
which is in contradiction with the fact that $c\in\T$. Therefore, we conclude that $\phi(\xz,\yo)=\phi(\xz,\yt)$ for all $\yo,\yt\in\PB$. Consequently, we may and do write $\phi(\xz,\y)=\phi(\xz)$
for $\y\in\PB$. It follows from \eqref{lem5.5.2} that $T_0(f)(\xz)=c\left\langle f(\phi(\xz))\right\rangle^{\ez(\xz,\y)}$ for all $f\in\A$ and $\y\in\PB$. By virtue of \eqref{(7c)}, there exists $f_1\in\A$ such that $f_1(\phi(\xz))=i$. Substituting $f=f_1$ into the preceding equation yields the following:
\[
T_0(f_1)(\xz)=c\left\langle f_1(\phi(\xz))\right\rangle^{\ez(\xz,\y)}=c\ez(\xz,\y)i
\]
for all $\y\in\PB$. This implies that the function $\ez(\xz,\y)$ is constant with respect to $\y\in\PB$. We shall henceforth write $\ez(\xz,\y)=\ez(\xz)$. Since $\xz\in\PA$ was chosen arbitrarily,
we conclude that $T_0(f)(\x)=c\left\langle f(\phi(\x))\right\rangle^{\ez(\x)}$ for all $f\in\A$ and $\x\in\PA$.
The maps $\phi$ and $\ez$, defined on $\PA$, are both continuous
on $\PA$ by definition.
\qed

\begin{claim}\label{cl2}
The sets $\PAp=\set{\x\in\PA:\ez(\x)=1}$ and $\PAm=\set{\x\in\PA:\ez(\x)=-1}$ are both closed and open subsets of $\PA$ such that $\PAp\cup\PAm=\PA$ and $\PAp\cap\PAm=\emptyset$.
\end{claim}

It suffices to note that the function $\ez$, defined on the set $\PA$, is continuous by Lemma~\ref{lem4.2} and Claim~\ref{lem5.6}.
\qed

\begin{claim}\label{cl8}
For each $\yz\in\PB$, the map
$\psi(\cdot,\yz)\colon\PA\to\PB$,
which maps $\x\in\PA$ to $\psi(\x,\yz)$,
is a constant map.
For the sake of brevity,
we shall henceforth write
$\psi(\x,\y)=\psi(\y)$
for $\x\in\PA$
and $\y\in\PB$.
This yields the following equation
for all $f\in\A$, $\x\in\PA$ and $\y\in\PB$:
\begin{equation}\label{cl8.1}
d(T_0(f))(\y)
=c^{\et(\x,\y)}\langle\varphio(\x,\y)d(f)(\psi(\y))\rangle^{\eo(\x,\y)}.
\end{equation}
In addition, the map $\psi\colon\PB\to\PB$
is continuous on $\PB$.
\end{claim}

Let us fix $\yz\in\PB$ arbitrarily.
We shall prove that $\psi(\xo,\yz)=\psi(\xt,\yz)$
for all $\xo,\xt\in\PA$.
Let us assume, for the sake of contradiction,
that there exist two points $\xo,\xt\in\PA$
such that $\psi(\xo,\yz)\neq\psi(\xt,\yz)$.
By virtue of condition \eqref{(7b)},
there exists $h_1\in\A$ such that
$d(h_1)(\psi(\xo,\yz))=1$ and $d(h_1)(\psi(\xt,\yz))=0$.
By combining these two equations with \eqref{lem5.5.2},
we obtain the conclusion that
\begin{multline*}
c^{\et(\xo,\yz)}\langle\varphio(\xo,\yz)\rangle^{\eo(\xo,\yz)}
=c^{\et(\xo,\yz)}\langle\varphio(\xo,\yz)d(h_1)(\psi(\xo,\yz))\rangle^{\eo(\xo,\yz)}\\
=d(T_0(h_1))(\yz)
=c^{\et(\xt,\yz)}\langle\varphio(\xt,\yz)d(h_1)(\psi(\xt,\yz))\rangle^{\eo(\xt,\yz)}
=0.
\end{multline*}
This is in contradiction with the fact that
$|c^{\et(\xo,\yz)}|=1=|\langle\varphio(\xo,\yz)\rangle^{\eo(\xo,\yz)}|$.
Therefore, we conclude that
$\psi(\xo,\yz)=\psi(\xt,\yz)$
for all $\xo,\xt\in\PA$.
We shall express $\psi(x,\yz)$ as $\psi(\yz)$
for all $\x\in\PA$.
We deduce from \eqref{lem5.5.2} that
\[
d(T_0(f))(\yz)
=c^{\et(\x,\yz)}\langle\varphio(\x,\yz)d(f)(\psi(\yz))\rangle^{\eo(\x,\yz)}
\]
for all $f\in\A$ and $\x\in\PA$.
By definition, it can be verified that
the map $\psi$, defined on $\PB$, is continuous
on $\PB$.
\qed

\begin{claim}\label{cl9}
For each $\yz\in\PB$, the maps
$c^{\et(\cdot,\yz)}\langle\varphio(\cdot,\yz)\rangle^{\eo(\cdot,\yz)}
\colon\PA\to\T$ and
$\eo(\cdot,\yz)\colon\PA\to\set{\pm1}$,
which map $\x\in\PA$ to
$c^{\et(\x,\yz)}\langle\varphio(\x,\yz)\rangle^{\eo(\x,\yz)}$
and $\eo(\x,\yz)$, respectively,
are constant maps.
For simplicity in notation,
we shall henceforth write
$c^{\et(\x,\y)}\langle\varphio(\x,\y)\rangle^{\eo(\x,\y)}=u(\y)$
and $\eo(\x,\y)=\eo(\y)$
for $\x\in\PA$ and $\y\in\PB$.
This yields the following equation
for all $f\in\A$ and $\y\in\PB$:
\begin{equation}\label{cl9.1}
d(T_0(f))(\y)=u(y)\langle d(f)(\psi(\y))\rangle^{\eo(\y)}.
\end{equation}
Moreover, the maps $u\colon\PB\to\T$
and $\eo\colon\PB\to\set{\pm1}$
are continuous on $\PB$.
\end{claim}

Let us fix an arbitrary $\yz\in\PB$.
By applying \eqref{(7b)}, there exists $h_2\in\A$ such that
$d(h_2)(\psi(\yz))=1$.
Upon substituting $f=h_2$ into equation \eqref{cl8.1}
results in the conclusion that
$d(T_0(h_2))(\yz)
=c^{\et(\x,\yz)}\langle\varphio(\x,\yz)\rangle^{\eo(\x,\yz)}$
for all $x\in\PA$.
This shows that the value
$c^{\et(\x,\yz)}\langle\varphio(\x,\yz)\rangle^{\eo(\x,\yz)}$
is independent of the variable $\x\in\PA$.
By defining $u(\yz)=c^{\et(\x,\yz)}\langle\varphio(\x,\yz)\rangle^{\eo(\x,\yz)}$,
we derive from equation \eqref{cl8.1}
that the following result holds for all $f\in\A$ and $x\in\PA$:
\[
d(T_0(f))(\yz)
=u(\yz)\langle d(f)(\psi(\yz))\rangle^{\eo(\x,\yz)}.
\]
It is possible to select $h_3\in\A$ such that
$d(h_3)(\psi(\yz))=i$.
The result is a direct consequence of condition \eqref{(7b)}
and the complex linearity of the mapping $d$.
Subsequent to the aforementioned equation,
it can be deduced that
$d(T_0(h_3))(\yz)=u(\yz)\eo(\x,\yz)i$
for all $\x\in\PA$.
This shows that $\eo(\x,\yz)$ is independent
of the variable $\x\in\PA$.
We shall henceforth write $\eo(\x,\yz)=\eo(\yz)$
for all $\x\in\PA$.
Since $\yz\in\PB$ was chosen arbitrarily,
we may regard $u$ and $\eo$ as mappings
from $\PB$ to $\T$ and $\PB$ to $\set{\pm1}$,
respectively.
Therefore, we conclude that
$d(T_0(f))(\y)=u(\y)\langle d(f)(\psi(\y))\rangle^{\eo(\y)}$
for all $f\in\A$ and $\y\in\PB$.
It can be seen that $u$ and $\eo$ are
continuous maps by definition.
\qed

\begin{claim}\label{cl3}
In light of Lemma \ref{lem4.4} and Claims \ref{lem5.3} and \ref{lem5.6}, the map $\phi$, defined on the set $\PA$, is a homeomorphism onto $\PA$.
\end{claim}

In accordance with Definition~\ref{defn2}, the map $\phi$ is surjective and continuous. We shall prove that $\phi$, defined on $\PA$, is injective. For $\xo,\xt\in\PA$ with $\xo\neq\xt$, we can choose $f_0\in\A$ so that $f_0(\xo)=1$ and $f_0(\xt)=0$ by \eqref{(7a)}. Since $T_0$ is surjective, there exists $f_1\in\A$ such that $T_0(f_1)=f_0$. Consequently, $T_0(f_1)(\xo)=1$ and $T_0(f_1)(\xt)=0$. By applying \eqref{lem5.6.1}, we obtain the following two equations:
\begin{align*}
&1=T_0(f_1)(\xo)=c\left\langle f_1(\phi(\xo))\right\rangle^{\ez(\xo)},\\
&0=T_0(f_1)(\xt)=c\left\langle f_1(\phi(\xt))\right\rangle^{\ez(\xt)}.
\end{align*}
Therefore, $f_1(\phi(\xo))\neq0=f_1(\phi(\xt))$. This implies that $\phi(\xo)\neq\phi(\xt)$, and thus $\phi$ is injective. It follows that $\phi$ is a bijective and continuous map from the compact space $\PA$ onto the Hausdorff space $\PA$. Consequently, it must be a homeomorphism.
\qed

\begin{claim}\label{cl10}
The sets $\PBp=\set{\y\in\PB:\eo(\y)=1}$ and $\PBm=\set{\y\in\PB:\eo(\y)=-1}$ are both closed and open subsets of $\PB$ such that $\PBp\cup\PBm=\PB$ and $\PBp\cap\PBm=\emptyset$.
\end{claim}

This is a direct consequence of the fact that the function $\eo$, defined on the set $\PB$, is continuous by Lemma~\ref{lem4.6} and Claim~\ref{cl9}.
\qed

\begin{claim}\label{cl12}
The map $\psi$, defined on the set $\PB$, is a homeomorphism onto $\PB$.
\end{claim}

It is evident that the map $\psi\colon\PB\to\PB$
is a surjective continuous map,
as asserted in Definition~\ref{defn2}
(cf. also Lemma~\ref{lem4.5}, Claims~\ref{lem5.4} and \ref{cl8}).
The following proof will demonstrate
that $\psi$ is injective.
For each $\yo,\yt\in\PB$ with $\yo\neq\yt$,
there exists $f_0\in\A$ such that
$d(f_0)(\yo)=1$ and $d(f_0)(\yt)=0$
by \eqref{(7b)}.
It is then possible to select $f_1\in\A$ such that
$T_0(f_1)=f_0$, since $T_0$ is surjective.
It follows that $d(T_0(f_1))(\yo)=1$ and
$d(T_0(f_1))(\yt)=0$.
Subsequently, we deduce from \eqref{cl9.1} that
\begin{align*}
1
&=
d(T_0(f_1))(\yo)
=u(\yo)\langle d(f_1)(\psi(\yo))\rangle^{\eo(\yo)},\\
0
&=
d(T_0(f_1))(\yt)
=u(\yt)\langle d(f_1)(\psi(\yt))\rangle^{\eo(\yt)}.
\end{align*}
This yields the result that
$d(f_1)(\psi(\yo))\neq0=d(f_1)(\psi(\yt))$
since $|u|=1$ on $\PB$ by Claim~\ref{cl9}.
This shows that $\psi(\yo)\neq\psi(\yt)$,
thereby demonstrating the assertion
that $\psi$ is injective.
Consequently, $\psi$ is a bijective continuous map
from the compact space $\PB$ onto
the Hausdorff space $\PB$.
Therefore, it must be a homeomorphism.
\qed

\begin{claim}\label{cl4}
For any $f\in\A$, $\x\in\PA$
and $\y\in\PB$, we have
\begin{align*}
&T(f)(\x)-T(0)(\x)=
\begin{cases}
cf(\phi(\x))&\x\in\PAp\\[1mm]
c\ov{f(\phi(\x))}&\x\in\PA\setminus\PAp
\end{cases},\\
&
d(T(f)-T(0))(\y)=
\begin{cases}
u(\y)d(f)(\psi(\y))&\y\in\PBp\\[1mm]
u(\y)\ov{d(f)(\psi(\y))}&\y\in\PB\setminus\PBp
\end{cases}.
\end{align*}
\end{claim}

These equations follow immediately from \eqref{T_0}, \eqref{lem5.6.1} and \eqref{cl9.1}.
\qed

\begin{claim}
Suppose that there exist a constant $c\in\T$,
homeomorphisms $\phi\colon\PA\to\PA$,
$\psi\colon\PB\to\PB$ and continuous functions
$\ez\colon\PA\to\set{\pm1}$, $\eo\colon\PB\to\set{\pm1}$
and $u\colon\PB\to\T$ such that $T_0$ satisfies the following
equations for all $f\in\A$, $\x\in\PA$ and $\y\in\PB$:
\[
T_0(f)(\x)=c\langle f(\phi(\x))\rangle^{\ez(\x)}
\qq\mbox{and}\qq
d(T_0(f))(\y)=u(\y)\langle d(f)(\psi(\y))\rangle^{\eo(\y)}.
\]
Then $T$ must be an isometry.
\end{claim}

We derive from the above equations
that the following results hold
for all $f,g\in\A$:
\begin{align*}
\V{T_0(f)-T_0(g)}
&=
\V{T_0(f)-T_0(g)}_\X+\V{d(T_0(f)-T_0(g))}_\Y\\
&=\sup_{\x\in\PA}|T_0(f)(\x)-T_0(g)(\x)|
+\sup_{\y\in\PB}|d(T_0(f))(\y)-d(T_0(g))(\y)|\\
&=
\sup_{\x\in\PA}|c\langle f(\phi(\x))-g(\phi(\x))\rangle^{\ez(\x)}|
+\sup_{\y\in\PB}|u(\y)\langle d(f)(\psi(\y))-d(g)(\psi(\y))\rangle^{\eo(\y)}|\\
&=
\sup_{\x\in\PA}|f(\phi(\x))-g(\phi(\x))|
+\sup_{\y\in\PB}|d(f)(\psi(\y))-d(g)(\psi(\y))|\\
&=\V{f-g}_\X+\V{d(f-g)}_\Y
=\V{f-g},
\end{align*}
where we have used the fact that $\phi$ and $\psi$ are
both surjective maps defined on the boundaries
$\PA$ and $\PB$ of $\A$ and $\B$, respectively.
Consequently, we conclude that $T_0$ is an isometry,
and therefore, $T=T_0+T(0)$ is an isometry.
This completes the proof of Theorem~\ref{thm1}.
\qed

\subsection{Proof of Corollary 2}

Let $\X$ be a one point set, written as $\set{\xz}$. We obtain the following by \eqref{lem5.6.1} and \eqref{cl9.1}:
\begin{align*}
&T_0(f)(\xz)=c\left\langle f(\phi(\xz))\right\rangle^{\ez(\xz)},\\
&d(T_0(f))(\y)=u(\y)\left\langle d(f)(\psi(\y))\right\rangle^{\eo(\y)}
\end{align*}
for all $f\in\A$ and $\y\in\PB$.
We may conclude that $\PA=\set{\xz}$
since $\emptyset\neq\cha\subset\PA\subset\X=\set{\xz}$.
Consequently, we obtain $\phi(\xz)=\xz$.
It follows that $T_0(f)(\xz)=cf(\xz)$ for all $f\in\A$ if $\ez(\xz)=1$, and $T_0(f)(\xz)=c\ov{f(\xz)}$ for all $f\in\A$ if $\ez(\xz)=-1$.
\qed

\subsection{Examples}
\begin{enumerate}
\item Let $C^1([0,1])$ be the complex linear space of all continuously differentiable complex-valued functions defined on the closed unit interval $[0,1]$. We set the following:
\begin{center}
$\K=\Y=[0,1],
\q
\A=C^1([0,1]),
\q
\B=C([0,1]).$
\end{center}
The mapping $d\colon\A\to\B$ is defined as $d(f)=f'$ for $f\in\A$. As shown in \cite[Lemma~1.6]{kaw}, the conditions \eqref{(5)} and \eqref{(6)} are satisfied. Since $\A$ contains polynomials, it follows that \eqref{(7)} holds. Two norms, $\VS{\cdot}$ and $\Vs{\cdot}$, are defined as follows: $\VS{f}=\Vinf{f}+\Vinf{f'}$ and $\Vs{f}=|f(0)|+\Vinf{f'}$ for $f\in C^1([0,1])$, where $\Vinf{\cdot}$ denotes the supremum norm on $[0,1]$. The characterization of surjective isometries on the Banach spaces $(C^1([0,1]),\VS{\cdot})$ and $(C^1([0,1]),\Vs{\cdot})$ was provided in \cite{kaw,miu1}. These characterizations of surjective isometries are special cases of Theorem~\ref{thm1} and Corollary~\ref{cor2}:
\begin{enumerate}
\item Let $\X=[0,1]$. Then $\VA{f}=\VX{f}+\VY{d(f)}=\VS{f}$ for $f\in C^1([0,1])$.
\item Let $\X=\set{0}$. Then $\VA{f} =|f(0)|+\Vinf{f'}=\Vs{f}$ for $f\in C^1([0,1])$.
\end{enumerate}

\item Let $\mathcal{A}$ be the complex linear space of all analytic functions on the open unit disc $\Di$, which can be extended to continuous functions on the closed unit disc $\Db$. For each $f\in\mathcal{A}$, we denote by $\hat{f}$ the unique continuous extension of $f$ to $\Db$. We define $\mathcal{S}_{\mathcal{A}} =\set{f\in\mathcal{A}:f'\in\mathcal{A}}$. We set
\begin{center}
$K=Y=\Db,
\q
\A=\set{\hat{f}\in C(\Db):f\in\mathcal{S_A}},
\q
\B=\set{\hat{g}\in C(\Db):g\in\mathcal{A}}.$
\end{center}
Let $d\colon\A\to\B$ be a map defined as follows: $d(\hat{f})=\wh{f'}$ for $\hat{f}\in\A$.
\begin{enumerate}
\item If $\X=\Db$, then \cite[Proposition~2.1, Lemmas~2.2 and 2.3]{miu3} show\rsout{s} that $\A$ satisfies the conditions \eqref{(5)}, \eqref{(6)} and \eqref{(7)}. Then $\VLA{\hat{f}}=\VL{\hat{f}}_{\Db}+\VL{\wh{f'}}_{\Db}$ for $\hat{f}\in\A$, where $\V{\cdot}_{\Db}$ represents the supremum norm on $\Db$. The characterization of surjective isometries from $(\A,\VA{\cdot})$ onto itself is provided by Theorem~\ref{thm1}.
While \cite{miu3} described surjective isometries from $(\mathcal{S_A},\VS{\cdot})$ onto itself, where $\VS{f}=\V{f}_{\Di}+\V{f'}_{\Di}$ for $f\in\A$.
\item If $\X=\set{\xz}$, then $\VLA{\hat{f}}=|f(\xz)|+\VL{\wh{f'}}_{\Db}$ for $\hat{f}\in\A$. The form of surjective isometries from $(\A,\VA{\cdot})$ onto itself is obtained by Corollary~\ref{cor2}. The characterization of surjective isometries between $(\mathcal{S_A},\Vs{\cdot})$ was presented in \cite{miu2}, where $\Vs{f}=|f(0)|+\V{f'}_{\Di}$ for $f\in\mathcal{S_A}$.
\end{enumerate}

\item Let $H^\infty$ be the commutative Banach algebra of all bounded analytic functions on the open unit disc. We define $\mathcal{S}^\infty =\set{f\in\mathcal{A}:f'\in H^\infty}$, where $\mathcal{A}$ is defined as in the previous example. The maximal ideal space of $H^\infty$ is denoted by $\M$. The Gelfand transform of $H^\infty$, denoted by $\wh{H}^\infty$, is a closed subalgebra of $C(\M)$. We set
\begin{center}
$K=Y=\M,
\q
\A=\set{\hat{f}\in C(\T):f\in\mathcal{S}^\infty},
\q
\B=\wh{H}^\infty$,
\end{center}
where $\hat{f}$ denotes the Gelfand transform of $f$. Let us define a map $d\colon\A\to\B$ as $d(\hat{f})=\wh{f'}$ for $\hat{f}\in\A$. It can be seen that the conditions \eqref{(5)}, \eqref{(6)} and \eqref{(7)} are satisfied by $\A$, as demonstrated in \cite[Lemmas~2.7, 2.10 and Propositions~3.4, 3.5]{miu4}.
\begin{enumerate}
\item If $\X=\T$, then $\VLA{\hat{f}}=\VL{\hat{f}}_{\T}+\VL{\wh{f'}}_{\M}$ for $\hat{f}\in\A$. The form of surjective isometries on $(\A,\VA{\cdot})$ is  given by Theorem~\ref{thm1}. The description of surjective isometries on $(\mathcal{S}^\infty,\VS{\cdot})$ can be found in \cite{miu4}, where $\VS{f}=\V{f}_{\Di}+\V{f'}_{\Di}$ for $f\in\mathcal{S}^\infty$.

\item If $\X=\set{\xz}$
for some $\xz\in\M$, then $\VLA{\hat{f}}=|f(\xz)|+\VL{\wh{f'}}_{\M}$ for $\wh{f}\in\A$.
Corollary~\ref{cor2} provides a characterization of surjective isometries on $(\A,\VA{\cdot})$,
while \cite{miu4} presented the form of surjective isometries
on $(\mathcal{S}^\infty,\Vs{\cdot})$, where $\Vs{f}=|f(0)|+\V{f'}_{\Di}$ for $f\in\mathcal{S}^\infty$.
\end{enumerate}
\end{enumerate}



\begin{thebibliography}{99} 
\bibitem{ban}S.~Banach, Theory of linear operations, Translated by F.~Jellett, Dover Publications, Inc. Mineola, New York, 2009.
\bibitem{bro}A.~Browder, Introduction to function algebras, W.A. Benjamin Inc. New York, 1969. 
\bibitem{fle1}R.~Fleming and J.~Jamison, Isometries on Banach spaces: function spaces, Chapman \& Hall/CRC Monographs and Surveys in Pure and Applied Mathematics, 129. Chapman \& Hall/CRC, Boca Raton, FL, 2003. 
\bibitem{hat1} O.~Hatori, \textit{The Mazur--Ulam property for a Banach space which satisfies a separation condition}, RIMS K\^{o}ky\^{u}roku Bessatsu \textbf{93} (2023), 29--82.
\bibitem{kaw}K.~Kawamura, H.~Koshimizu and T.~Miura, \textit{Norms on $C^1([0,1])$ and their isometries}, Acta Sci. Math. (Szeged) \textbf{84} (2018), 239--261. 
\bibitem{maz}S.~Mazur and S.~Ulam, \textit{Sur les transformationes isom\'{e}triques d'espaces  vectoriels norm\'{e}s},  C. R. Acad. Sci. Paris \textbf{194} (1932), 946--948. 
\bibitem{miu1}T.~Miura and H.~Takagi, \textit{Surjective isometries on the Banach space of continuously differentiable functions}, Contemp. Math. \textbf{687} (2017), 181--192.
\bibitem{miu2}T.~Miura and N.~Niwa, \textit{Surjective isometries on a Banach space of analytic functions on the open unit disc}, Nihonkai Math. J. \textbf{29} (2018), 53--67.
\bibitem{miu3}T.~Miura and N.~Niwa, \textit{Surjective isometries on a Banach space of analytic functions on the open unit disc, II}, Nihonkai Math. J. \textbf{31} (2020), 75--91.
\bibitem{miu4}T.~Miura and N.~Niwa, \textit{Surjective isometries on a Banach space of analytic functions with bounded derivatives}, Acta Sci. Math. \textbf{89} (2023), 109--145.
\bibitem{nic}B.~Nica, \textit{The Mazur--Ulam theorem}, Expo. Math. \textbf{30} (2012), 397--398.
\bibitem{nov}W.P.~Novinger and D.M.~Oberlin, \textit{Linear isometries of some normed spaces of analytic functions}, Can. J. Math. \textbf{37} (1985), 62--74.
\bibitem{rao}N.V.~Rao and A.K.~Roy, \textit{Linear isometries of some function spaces}, Pacific J. Math. \textbf{38} (1971), 177--192. 
\bibitem{rud}W.~Rudin, Real and complex analysis. Third edition. McGraw-Hill Book Co., New York, 1987.
\end{thebibliography}
\end{document}